\documentclass[11pt,a4paper]{article}

\usepackage{amssymb}
\usepackage{amsmath}
\usepackage{amsthm}
\usepackage{mathrsfs}
\usepackage{tikz}
\usepackage[T1]{fontenc}
\usepackage{inputenc}
\usepackage[english]{babel}
\usepackage{lmodern}
\usepackage{geometry}
\usepackage{changepage}
\changepage{0pt}{}{}{}{}{0pt}{}{0pt}{10pt}
\usepackage[numbers]{natbib}
\usepackage{hyperref}
\hypersetup{
pdfpagemode=UseNone,
pdftoolbar=true,        
pdfmenubar=true,        
pdffitwindow=false,     
pdfstartview={Fit},    
pdftitle={Team organization may help swarms of flies to become invisible in closed waveguides},    
pdfauthor={L. Chesnel, S.A. Nazarov},     
pdfcreator={L. Chesnel, S.A. Nazarov},   
pdfproducer={L. Chesnel, S.A. Nazarov}, 
pdfkeywords={}, 
pdfnewwindow=true,      
colorlinks=true,       
linkcolor=magenta,          
citecolor=red,        
filecolor=cyan,      
urlcolor=blue           
}

\newcommand{\dsp}{\displaystyle}

\newcommand{\eps}{\varepsilon}
\newcommand{\om}{\omega}
\newcommand{\Om}{\Omega}
\newcommand{\mrm}[1]{\mathrm{#1}}

\newcommand{\Cplx}{\mathbb{C}}
\newcommand{\N}{\mathbb{N}}
\newcommand{\R}{\mathbb{R}}

\newcommand{\Capa}{\mrm{cap}}
\newcommand{\mH}{\mrm{H}}

\newcommand{\mW}{\mrm{W}}

\newcommand{\mX}{\mrm{X}}
\newcommand{\loc}{\mbox{\scriptsize loc}}
\newcommand{\pro}{\mbox{\scriptsize pro}}
\newcommand{\expo}{\mbox{\scriptsize exp}}

\newtheorem{lemma}{Lemma}[section]
\newtheorem{remark}{Remark}[section]

\newtheorem{proposition}{Proposition}[section]

\begin{document}

~\vspace{0.3cm}
\begin{center}
{\sc \bf\LARGE 
Team organization may help swarms of flies\\[6pt]  to become invisible in closed waveguides}
\end{center}

\begin{center}
\textsc{Lucas Chesnel}$^1$, \textsc{Sergei A. Nazarov}$^{2,\,3,\,4}$\\[16pt]
\begin{minipage}{0.92\textwidth}
{\small
$^1$ INRIA/Centre de mathématiques appliquées, \'Ecole Polytechnique, Université Paris-Saclay, Route de Saclay, 91128 Palaiseau, France;\\
$^2$ Faculty of Mathematics and Mechanics, St. Petersburg State University, Universitetskaya naberezhnaya, 7-9, 199034, St. Petersburg, Russia;\\
$^3$ Laboratory for mechanics of new nanomaterials, St. Petersburg State Polytechnical University, Polytekhnicheskaya ul, 29, 195251, St. Petersburg, Russia;\\
$^4$ Laboratory of mathematical methods in mechanics of materials, Institute of Problems of Mechanical Engineering, Bolshoy prospekt, 61, 199178, V.O., St. Petersburg, Russia.\\[10pt] 
E-mails:  \texttt{lucas.chesnel@inria.fr},
\texttt{srgnazarov@yahoo.co.uk}\\[-14pt]
\begin{center}
(\today)
\end{center}
}
\end{minipage}
\end{center}
\vspace{0.4cm}

\noindent\textbf{Abstract.} 
We are interested in a time harmonic acoustic problem in a waveguide containing flies. The flies are modelled by small sound soft obstacles. We explain how they should arrange to become invisible to an observer sending waves from $-\infty$ and measuring the resulting scattered field at the same position. We assume that the flies can control their position and/or their size. Both monomodal and multimodal regimes are considered. On the other hand, we show that any sound soft obstacle (non necessarily small) embedded in the waveguide always produces some non exponentially decaying scattered field at $+\infty$ for wavenumbers smaller than a constant that we explicit. As a consequence, for such wavenumbers, the flies cannot be made completely invisible to an observer equipped with a measurement device located at $+\infty$.\\


\section{Introduction}\label{Introduction}
Recently, questions of invisibility in scattering theory have drawn much attention. In particular, hiding objects is an activity in vogue. In this direction, the development of metamaterials with exotic physical parameters has played a fundamental role allowing the realization of cloaking devices. One of the most popular techniques which has been proposed consists of surrounding the object to hide by a well-chosen material so that waves go through as if there was no scatterer. In this approach, simple concepts of transformation optics allow one to determine the \textit{ad hoc} material constituting the cloaking device (see e.g. \cite{CPSSP06,GKLU09,ChCS10}). It is important to emphasize that complex materials whose physical parameters exhibit singular values are required to build the device. From a practical point of view, constructing such materials is a challenging problem that people have not yet been able to solve. \\
\newline
For some applications, one may want to build invisible objects. But for others, it is better if they do not exist. In particular, for imaging methods, it is preferable that two different settings provide two different sets of measurements so that one can hope to recover features of the probed medium. In this field, invisible objects are interesting to study to understand the limits of a given technique. Indeed, it is important to have an idea of which objects can be reconstructed and which one cannot to assess how robust the existing algorithms are. Let us mention also that some techniques, like the Linear Sampling Method in inverse scattering theory, work only when invisible scatterers do not exist \cite{CoKi96,CoPP97,BoLu08,ArGL11,CaHa13b}.\\
\newline
In the present article, we consider a scattering problem in a closed waveguide, that is in a waveguide which has a bounded transverse section. In such a geometry, at a given frequency, only a finite number of waves can propagate. We are interested in a situation where an observer wants to detect the presence of defects in some reference waveguide from far-field backscattering data. Practically, the observer sends waves, say from $-\infty$, and measures the amplitude of the resulting scattered field at the same position. It is known that at $-\infty$, the scattered field decomposes as the sum of a finite number of propagative waves plus some exponentially decaying remainder. We shall say that the defects are invisible if for all incident propagative waves, the resulting scattered field is exponentially decaying at $-\infty$. In this setting, examples of invisible obstacles, obtained via numerical simulations, can be found in literature. We refer the reader to \cite{EvMP14} for a water waves problem and to \cite{AlSE08,EASE09,NgCH10,OuMP13,FuXC14} for strategies using new ``zero-index'' and ``epsilon near zero'' metamaterials in electromagnetism (see also \cite{FlAl13} for an application to acoustics). The technique that we propose in this article differs from the ones presented in the above mentioned works because it is exact in the sense that it is a rigorous proof of existence of invisible obstacles.\\ 
\newline
We will work with small sound soft obstacles of size $\eps$ (as in the so-called MUSIC algorithm \cite{Ther92,Chen01,Kirs02}) that we call flies in the rest of the paper. To simplify the presentation, we assume here\footnote{Actually, we will make this assumption everywhere in the article except in \S\ref{paraSeveralWave}.} that the frequency is such that the observer can send only one wave and measures one reflection coefficient $s^{\eps-}$ (the amplitude of the scattered field at $-\infty$). With this notation, our goal is to impose $s^{\eps-}=0$. The approach we propose to help flies to become invisible is based on the following basic observation: when there is no obstacle in the waveguide, the scattered field is null so that $s^{0-}=0$. When flies of size $\eps$ are located in the waveguide, we can prove that the reflection coefficient $s^{\eps-}$ is of order $\eps$. Our strategy, which is inspired from \cite{Naza11,BoNa13}, consists in computing an asymptotic (Taylor) expansion of $s^{\eps-}$ as $\eps$ tends to zero. In this expansion, the first terms have a relatively simple and explicit dependence with respect to the features of the flies (position and shape). This is interesting because it allows us to use theses parameters as control terms to cancel the \textit{whole} expansion of $s^{\eps-}$ (and not only the first term obtained with the Born approximation). More precisely, slightly perturbing the position or the size of one or several flies, it is possible to introduce some new degrees of freedom that we can tune to impose $s^{\eps-}=0$. We underline that in principle, the sound soft obstacles have to be small compared to the wavelength. This explains the introduction of diptera terminology. \\ 
\newline
The technique described above mimics the proof of the implicit function theorem. It has been introduced in \cite{Naza11,Naza11c,Naza12,Naza13,CaNR12,Naza11b} with the concept of ``enforced stability for embedded eigenvalues''. In these works, the authors develop a method for constructing small regular and singular perturbations of the walls of a waveguide that preserve the multiplicity of the point spectrum on a given interval of the continuous spectrum. The approach has been adapted in \cite{BoNa13,BLMN15} (see also \cite{BoNTSu,BoCNSu,ChHS14} for applications to other problems) to prove the existence of regular perturbations of a waveguide, for which several waves at given frequencies pass through without any distortion or with only a phase shift. In the present article, the main novelty lies in the fact that we play with small obstacles and not with regular perturbations of the wall of the waveguide to achieve invisibility. This changes the asymptotic expansion of the reflection coefficient $s^{\eps-}$ and we can not cancel it exactly as in \cite{BoNa13}. In particular, as we will observe later (see Remark \ref{RmqSwarm}), the flies have to act as a team to become invisible: a single fly cannot be invisible. To some extent, our work shares similarities with the articles \cite{Naza10,CaNP10b,Naza12b} where the authors use singular perturbations of the geometry to open gaps in the spectrum of operators considered in periodic waveguides.\\
\newline
In this article, we consider a scattering problem in a closed waveguide with a finite number of propagative waves. Note that the analysis we will develop can be easily adapted to construct sound soft obstacles in freespace which are invisible to an observer sending incident plane waves and measuring the far field pattern of the resulting scattered field in a finite number of directions (setting close to the one of  \cite{BoCNSu}).\\
\newline
The text is organized as follows. In the next section, we describe the setting and introduce adapted notation. In Section \ref{sectionAsympto} we compute an asymptotic expansion of the field 
$u^{\eps}$, the solution to the scattering problem in the waveguide containing small flies of size $\eps$, as $\eps$ tends to zero. There is a huge amount of literature concerning scattering by small obstacles (see, among other references, \cite{Foldy45,MaNP00,KaNa00,Ramm05,Mart06,CaHa13,BeCTPr}) and what we will do in this section is rather classical.  Then in Section \ref{sectionAsymptoCoeff}, from the expression of $u^{\eps}$, we derive an asymptotic expansion of the reflection coefficient $s^{\eps-}$ appearing in the decomposition of  $u^{\eps}$ at $-\infty$. In Section \ref{sectionPerturbation}, slightly modifying the position of one fly and solving a fixed point problem we explain how to cancel all the terms in the asymptotic expansion of $s^{\eps-}$ to impose $s^{\eps-}=0$. Proposition \ref{propositionMainResult} is the first main result of the paper. In Section \ref{sectionTransInv}, we study the question of invisibility assuming that the observer can send waves from $-\infty$ and measure the resulting scattered field at $+\infty$. More precisely, we show that it is impossible that the scattered field produced by the defect in the waveguide decays exponentially at $+\infty$ for wavenumbers $k$ smaller than a constant that can be explicitly computed. This result holds for all sound soft obstacles (not necessarily small) embedded in the waveguide. Proposition \ref{MainPropositionMixt} is the second main result of the paper. In Section \ref{SectionSizeOfFlies}, we come back to backscattering invisibility for flies and instead of playing with their position, we modify slightly their size. With this degree of freedom, we explain how to cancel the reflection coefficient. In a first step, we consider the case where there is only one propagative wave. Then we work at higher frequency with several propagative waves. In this setting, the higher the frequency is, the more information the observer can get. Quite logically, in our approach, we shall need more and more flies to cancel the different reflection coefficients as the number of propagative waves increases. We provide a short conclusion in Section \ref{SectionConclu}. Finally, Appendix \ref{JustiAsymptotics} is dedicated to proving technical results needed in the justification of asymptotic expansions obtained formally in Section \ref{sectionPerturbation}.

\section{Setting of the problem}\label{SectionWaveguide}

\begin{figure}[!ht]
\centering
\begin{tikzpicture}[scale=0.9]
\draw[dashed] (-3,1) ellipse (0.4 and 1);
\draw[fill=gray!5] (-2,1) ellipse (0.4 and 1);
\draw[fill=gray!5,draw=none](-2,0) rectangle (2,2);
\draw (-2,0)--(2,0);
\draw[dashed](-3,0)--(-2,0);
\draw[dashed](3,0)--(2,0);
\draw[dashed](-3,2)--(-2,2);
\draw[dashed](3,2)--(2,2);
\draw(-2,2)--(2,2);
\draw[fill=gray!50,fill opacity=0.4](0:-2 and 0) arc (90:270:0.4 and -1)--(2,2)--(2,0)--(-2,0);
\draw[fill=gray!40] (2,1) ellipse (0.4 and 1);
\node at (-1.5,0.3){\small$\Omega^0$};
\node at (-1.3,2.2){\small$\Gamma^0=\partial\Om^0$};
\draw[dashed] (3,1) ellipse (0.4 and 1);
\end{tikzpicture}\hspace{0.9cm}\begin{tikzpicture}[scale=0.9]
\node[inner sep=0pt] at (0.7,0.4)  {\includegraphics[width=8mm]{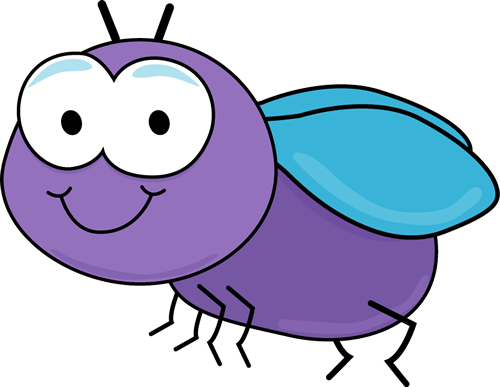}};
\node[inner sep=0pt] at (-1,1.2)  {\includegraphics[width=8mm]{mouche1}};
\draw[dashed] (-3,1) ellipse (0.4 and 1);
\draw[fill=gray!50,fill opacity=0.4](0:-2 and 0) arc (90:270:0.4 and -1)--(2,2)--(2,0)--(-2,0);
\node at (-1,1.65){\small$\mathcal{O}^\eps_1$};
\node at (0.7,0.85){\small$\mathcal{O}^\eps_2$};
\draw (-2,0)--(2,0);
\draw (-2,2)--(2,2);
\draw [dashed](-3,0)--(-2,0);
\draw [dashed](3,0)--(2,0);
\draw [dashed](-3,2)--(-2,2);
\draw [dashed](3,2)--(2,2);
\draw[fill=gray!40] (2,1) ellipse (0.4 and 1);
\draw[dashed] (3,1) ellipse (0.4 and 1);
\node at (-1.5,0.3){\small$\Omega^{\eps}$};
\node at (-1.3,2.2){\small$\Gamma^{\eps}=\partial\Om^{\eps}$};
\end{tikzpicture}\hspace{0.9cm}\begin{tikzpicture}
\node[inner sep=0pt] at (0,1.6)  {\includegraphics[width=2cm]{mouche1}};
\node at (0.3,2.4){\small$\mathcal{O}$};
\node at (0,1){\small \phantom{espa}};
\end{tikzpicture}\vspace{0.4cm}
\caption{Unperturbed waveguide (left), perturbed waveguide with two flies (middle), set $\mathcal{O}$ (right).\label{Domain}} 
\end{figure}

\noindent Let $\Om^{0}:=\{x=(y,z)\,|\,y\in\om\mbox{ and }z\in\R\}$ be a cylinder of $\R^3$. Here, $\om\subset \R^{2}$ is a connected open set with smooth boundary. In the following, $\Om^{0}$ is called the \textit{reference} or \textit{unperturbed} waveguide (see Figure \ref{Domain} on left). Let $\mathcal{O}\subset \R^3$ be an open set with Lipschitz boundary. Consider $M_1=(y_1,z_1)$, $M_2=(y_2,z_2)$ two points located in $\Om^{0}$ and define the sets, for $n=1,2$, $\eps>0$
\[
\mathcal{O}^\eps_{n} := \{x\in\R^3\,|\,\eps^{-1}(x-M_n)\in\mathcal{O} \}.
\]
Let $\eps_0>0$ denote a positive parameter such that $\overline{\mathcal{O}^\eps_{n}}\subset \Om^{0}$ for all $\eps\in(0;\eps_0]$, $n=1,2$. We call \textit{perturbed} waveguide (see Figure \ref{Domain} on right) the set
\begin{equation}\label{defSetIni}
\Om^{\eps} :=  \Om^{0}\setminus \bigcup_{n=1}^2\overline{\mathcal{O}^\eps_n}.
\end{equation}
The sets $\mathcal{O}^\eps_1$ and $\mathcal{O}^\eps_2$ model the flies located in the waveguide. To begin with, and to simplify the exposition we assume that there are only two of them and not a ``swarm''. For the latter configuration, we refer the reader to \S\ref{paraSeveralWave}. We are interested in the propagation of acoustic waves in time harmonic regime in $\Om^{\eps}$. Imposing soft wall boundary condition, it reduces to the study of the Dirichlet problem for the Helmholtz equation 
\begin{equation}\label{pbInit}
\begin{array}{|rcll}
-\Delta u  & = & k^2 u & \mbox{ in }\Om^{\eps}\\
 u  & = & 0  & \mbox{ on }\Gamma^{\eps}:=\partial\Om^{\eps}.
\end{array}
\end{equation}
Here, $u$ represents for example the pressure in the medium filling the waveguide, $k$ corresponds to the wavenumber proportional to the frequency of harmonic oscillations, $\Delta$ is the Laplace operator. Using separation of variables in the unperturbed waveguide $\Om^0$, it is easy to compute the solutions of the problem
\begin{equation}\label{pbInitRef}
\begin{array}{|rcll}
-\Delta u  & = & k^2 u & \mbox{ in }\Om^0\\
 u  & = & 0  & \mbox{ on }\Gamma^0:=\partial\Om^{0}.
\end{array}
\end{equation}
To provide their expression, let us introduce $\lambda_j$ and $\varphi_j$ the eigenvalues and the corresponding eigenfunctions of the Dirichlet problem for the Laplace operator on the cross-section $\om$ 
\begin{equation}\label{eigenpairTransverse}
\begin{array}{l}
0<\lambda_1 < \lambda_2 \le \lambda_3 \le \cdots \le \lambda_j \le \cdots \rightarrow +\infty,\\
(\varphi_j,\varphi_{j'})_{\om}=\delta_{j,j'},\qquad j,j'\in\N^{\ast}:=\{1,2,\dots\}. 
\end{array}
\end{equation}
Here, $\delta_{j,j'}$ stands for the Kronecker symbol. In this paper, for any measurable set $\mathscr{O}\subset\R^r$, $r\ge1$, we make no distinction between the complex inner products of the Lebesgue spaces $\mrm{L}^2(\mathscr{O})$ and $\mrm{L}^2(\mathscr{O})^r$, just using the notation $(\cdot,\cdot)_{\mathscr{O}}$. The fact that the first eigenvalue $\lambda_1$ is simple is a consequence of the Krein-Rutman theorem. Moreover, we know that $\varphi_1$ has a constant sign on $\om$ (see e.g. \cite[Thm. 1.2.5]{Henr06}).  Assume that $k\in\R$ is such that $k^2\ne \lambda_j$ for all $j\in\N^{\ast}$. Then, the solutions of (\ref{pbInitRef}), the \textit{modes} of the waveguide defined up to a multiplicative constant, are given by
\begin{equation}\label{defModes}
w^{\pm}_j(y,z)= (2|\beta_j|)^{-1/2} e^{\pm i \beta_j z}\varphi_j(y)\quad \mbox{ with }\ \beta_j:=\sqrt{k^2-\lambda_j}.
\end{equation}
All through the paper, the complex square root is chosen so that if $c=r e^{i\gamma}$ for $r\ge 0$ and $ \gamma \in[0;2\pi)$, then $\sqrt{c}=\sqrt{r}e^{i\gamma/2}$. With this choice, there holds $\Im m\,\sqrt{c}\ge 0$ for all $c\in\Cplx$. According to the value of $k^2$ with respect to the $\lambda_j$, the modes $w^{\pm}_j$ adopt different behaviours. 
For $j\in N_{\expo}:=\{j'\in \N^{\ast}\,|\,\lambda_{j'} >k^2\}$, the function $w^{+}_j$ (resp. $w^{-}_j$) decays exponentially at $+\infty$ (resp. $-\infty$) and grows exponentially at $-\infty$ (resp. $+\infty$). For $j\in N_{\pro}:=\{j'\in \N^{\ast}\,|\,\lambda_{j'} <k^2\}$, the functions $w^{\pm}_j$ are \textit{propagating waves} in $\Om^0$. In the present paper, except in \S\ref{paraSeveralWave}, we shall assume that the wavenumber $k$ verifies 
\begin{equation}\label{AssumptionWaveNber}
\lambda_1 < k^2 < \lambda_2.
\end{equation}
In this case, there are only two propagating waves $w^{\pm}_1$ and to simplify, we denote $w^{\pm}:=w^{\pm}_1$. In the perturbed  waveguide $\Om^{\eps}$, the wave $w^+$ travels from $-\infty$, in the positive direction of the $(Oz)$ axis and is scattered by the obstacles. In order to model this phenomenon, classically, it is  necessary to supplement equations (\ref{pbInit}) with proper radiation conditions at $\pm\infty$. Let us denote $\mrm{H}^1_{\loc}(\Om^{\eps})$ the set of measurable functions whose $\mH^1$-norm is finite on each bounded subset of $\Om^{\eps}$. We will say that a function $v\in\mrm{H}^1_{\loc}(\Om^{\eps})$ which satisfies equations (\ref{pbInit}) is \textit{outgoing} if it admits the decomposition
\begin{equation}\label{scatteredField}
v = \chi^+ s^{+} w^+ + \chi^- s^{-} w^- + \tilde{v}\ ,
\end{equation}
for some constants $s^{\pm}\in\Cplx$ and some $\tilde{v}\in\mH^1(\Om^{\eps})$. In (\ref{scatteredField}), $\chi^{+}\in\mathscr{C}^{\infty}(\Om_0)$ (resp. $\chi^{-}\in\mathscr{C}^{\infty}(\Om_0)$) is a cut-off function equal to one for $z\ge \ell$ (resp. $z\le -\ell$) and equal to zero for $z\le \ell/2$ (resp. $z\ge -\ell/2$). The constant $\ell>0$ is chosen large enough so that $\Om^{\eps}$ coincides with $\Om^{0}$ for $x=(y,z)$ such that $|z|\ge \ell/2$. Using Fourier decomposition, we can show that the remainder $\tilde{v}$ appearing in (\ref{scatteredField}) is exponentially decaying at $\pm\infty$. Now, the scattering problem we consider states 
\begin{equation}\label{PbChampTotalBIS}
\begin{array}{|rcll}
\multicolumn{4}{|l}{\mbox{Find }u^{\eps}\in\mrm{H}^1_{\loc}(\Om^{\eps}) \mbox{ such that }u^{\eps}-w^+\mbox{ is outgoing and } }\\[3pt]
-\Delta u^{\eps}  & = & k^2 u^{\eps} & \mbox{ in }\Om^{\eps}\\[3pt]
 u^{\eps}  & = & 0  & \mbox{ on }\Gamma^{\eps}.
\end{array}
\end{equation}
It is known that Problem (\ref{PbChampTotalBIS}), in an appropriate framework, satisfies the Fredholm alternative (see e.g. \cite[Chap. 5, \S3.3, Thm. 3.5 p. 160]{NaPl94}). Moreover, working by contradiction, adapting for example the proof of \cite[Lem. 3.1]{Clay09}, one can show that (\ref{PbChampTotalBIS}) admits a unique solution $u^{\eps}$ for $\eps$ small enough. In particular, there are no trapped modes for $\eps$ small enough. In the following, $u^{\eps}-w^+$ (resp. $u^{\eps}$) will be referred to as the \textit{scattered} (resp. \textit{total}) field associated with the \textit{incident} field $w^+$. We emphasize that in (\ref{PbChampTotalBIS}), $w^+$ is the source term. The coefficients $s^{\pm}$ appearing in (\ref{scatteredField}) with $v$ replaced by $u^{\eps}-w^+$ will be denoted $s^{\eps\pm}$, so that there holds 
\begin{equation}\label{DecompoChampScattered}
u^{\eps}-w^+ = \chi^+ s^{\eps+} w^+ + \chi^- s^{\eps-} w^- + \tilde{u}^{\eps},
\end{equation}
where $\tilde{u}^{\eps}\in\mH^1(\Om^{\eps})$ is a term which is exponentially decaying at $\pm\infty$. With this notation, the usual \textit{reflection} and \textit{transmission} coefficients are respectively given by 
\begin{equation}\label{DecompoChampTotal}
R^{\eps}=s^{\eps-}  \qquad \mbox{ and }\qquad T^{\eps} = 1+s^{\eps+}.
\end{equation}
Our goal is to explain how the flies $\mathcal{O}^{\eps}_{n}$, $n=1,2$, should arrange so that there holds $R^{\eps}=0$. In our analysis, we shall assume that the flies can play with their position or with their size. Note that when $s^{\eps-}=0$ (or equivalently when $R^{\eps}=0$), the scattered field $u^{\eps}-w^+$ defined from (\ref{PbChampTotalBIS}) is exponentially decaying at $-\infty$. As a consequence, an observer located at $z=-L$, with $L$ large, sending the wave $w^+$ and measuring the resulting scattered field is unable to detect the presence of the flies.

\section{Asymptotic expansion of the total field}\label{sectionAsympto}
In this section, we compute an asymptotic expansion of the solution $u^{\eps}$ to Problem (\ref{PbChampTotalBIS}) as $\eps$ tends to zero. The method to derive such an expansion is classical (see for example \cite[\S2.2]{MaNP00}) but we detail it for the sake of clarity. In accordance with the general theory of asymptotic analysis, we make the ansatz
\begin{equation}\label{DefAnsatz}
\begin{array}{lcl}
u^{\eps} & = & \phantom{+ \eps^m } u_{0} + \dsp\sum_{n=1}^{2}\zeta_n(x)\,v_{0,\,n}(\eps^{-1}(x-M_n)) \\[10pt]
 & & + \eps\phantom{l}\Big(\dsp u_{1} + \sum_{n=1}^{2}\zeta_n(x)\,v_{1,\,n}(\eps^{-1}(x-M_n))\Big) \\[10pt]
 & & + \eps^{2}\Big(\dsp u_{2} + \sum_{n=1}^{2}\zeta_n(x)\,v_{2,\,n}(\eps^{-1}(x-M_n))\Big)+\dots
\end{array}
\end{equation}
where the dots stand for terms of high order unnecessary in the study. In this ansatz, the functions $v_{k,\,n}$ correspond to boundary layer terms. They depend on the rapid variables $\eps^{-1}(x-M_n)$ and compensate the residual of the principal asymptotic terms $u_{k}$ in a neighbourhood of $M_n$, $n=1,2$. Moreover, for $n=1,2$, $\zeta_n\in\mathscr{C}^{\infty}_0(\R^3,[0;1])$ denotes a cut-off function which is equal to one in a neighbourhood of $M_n$ and whose support is a compact set sufficiently small so that $\zeta_n=0$ on $\Gamma^0$, $\zeta_n(M_m)=0$ in a neighbourhood of $M_m\ne M_n$. Now, we explain how to define each term in (\ref{DefAnsatz}). For justification of the asymptotic expansion and error estimates, we refer the reader to Section \ref{JustiAsymptotics}.\\
\newline 
$\star$ At order $\eps^0$, the incident wave $w^{+}$ does not see the small obstacles and there is no scattered field. Therefore, we take $u_{0}=w^+$. For $n=1,2$, the function $v_{0,\,n}$ allows us to impose the Dirichlet boundary condition on $\partial\mathcal{O}^{\eps}_{n}$ at order $\eps^{0}$. For $x\in\partial\mathcal{O}^{\eps}_{n}$, computing a Taylor expansion, we find $u_{0}(x)=u_{0}(M_n)+(x-M_n)\cdot\nabla u_{0}(M_n)+\dots\ $. Note that $x-M_n$ is of order $\eps$. To simplify notations, for $n=1,2$, we introduce the fast variable $\xi_n=\eps^{-1}(x-M_n)$. For the correction terms $v_{k,\,n}$, in a neighbourhood of $M_n$ (remember that the cut-off function $\zeta_n$ is equal to one in this region), we obtain 
\[
\begin{array}{ll}
& (\Delta_x+k^2\mrm{Id})\left(v_{0,\,n}(\eps^{-1}(x-M_n))+\eps\,v_{1,\,n}(\eps^{-1}(x-M_n))+\eps^2\,v_{2,\,n}(\eps^{-1}(x-M_n))
+\dots\right)\\[8pt]
 = &\eps^{-2}\,\Delta_{\xi_n}v_{0,\,n}(\xi_n)+\eps^{-1}\,\Delta_{\xi_n}v_{1,\,n}(\xi_n)+\eps^{0}\,(\Delta_{\xi_n}v_{2,\,n}(\xi_n)+k^2v_{0,\,n}(\xi_n))+\dots\ .
\end{array}
\]
Since there is no term of order $\eps^{-2}$ in the expansion (\ref{DefAnsatz}), we impose $\Delta_{\xi_n}v_{0,\,n}=0$ in $\R^3\setminus\overline{\mathcal{O}}$. This analysis leads us to take 
\begin{equation}\label{defSdnTerm0}
v_{0,\,n}(\xi_n) = -u_{0}(M_n)\,W(\xi_n).
\end{equation}
Here, $W$ is the capacity potential for $\mathcal{O}$ (\textit{i.e.} $W$ is harmonic in $\R^3\setminus\overline{\mathcal{O}}$, vanishes at infinity and verifies $W=1$ on $\partial\mathcal{O}$). In the sequel, the asymptotic behaviour of $W$ at infinity will play a major role. As $|\xi|\to+\infty$, we have (see e.g. \cite{Lank72})
\[
W(\xi) = \frac{\Capa(\mathcal{O})}{|\xi|} + \vec{q}\cdot\nabla\Phi(\xi) + O(|\xi|^{-3}),
\]
where $\Phi:=\xi\mapsto-1/(4\pi|\xi|)$ is the fundamental solution of the Laplace operator in $\R^3$ and $\vec{q}$ is some given vector in $\R^3$. The term $\Capa(\mathcal{O})$ corresponds to the harmonic capacity \cite{PoSz51} of the obstacle $\mathcal{O}$. Let us translate the position of the origin, making the change of variable $\xi\mapsto \xi^\theta:=\xi+\theta$, for a given $\theta\in\R^3$. When $|\xi|\to+\infty$, we can write 
\begin{equation}\label{CalculusChangeOfCoor}
\begin{array}{lcl}
W(\xi) & = & \dsp\frac{\Capa(\mathcal{O})}{|\xi|} + \vec{q}\cdot\nabla\Phi(\xi) + O(|\xi|^{-3})\\[20pt]
 & = & \dsp\frac{\Capa(\mathcal{O})}{|\xi^\theta-\theta|} + \frac{\vec{q}\cdot(\xi^\theta-\theta)}{4\pi|\xi^\theta-\theta|^3}+ O(|\xi^\theta-\theta|^{-3})\\[20pt]
 & = & \dsp\frac{\Capa(\mathcal{O})}{|\xi^\theta|} + \left(\frac{\vec{q}}{4\pi}+\Capa(\mathcal{O})\theta \right)\cdot\frac{\xi^\theta}{|\xi^\theta|^{3}} + O(|\xi^\theta|^{-3}).
\end{array}
\end{equation}
Since $\Capa(\mathcal{O})=\int_{\R^3\setminus\overline{\mathcal{O}}}|\nabla W|^2\,d\,\xi>0$ \cite{PoSz51}, this shows that there is exactly one value of $\theta\in\R^{3}$ such that $W$, in the new system of coordinates, admits the asymptotic expansion
\begin{equation}\label{FarFieldOfNearField}
W(\xi^\theta) = \dsp\frac{\Capa(\mathcal{O})}{|\xi^\theta|} + O(|\xi^\theta|^{-3}).
\end{equation}
In the following, we shall always assume that $W$ is defined in the system of coordinates centered at $O+\theta$ and to simplify, we shall denote  $\xi$ instead of $\xi^{\theta}$.
\begin{remark}
This trick to obtain a simple expansion for $W$ at infinity is not a necessary step in our procedure. However, it allows one to shorten the calculus. 
\end{remark}
\noindent $\star$ Now, we turn to the terms of order $\eps$ in the expansion of $u^{\eps}$. After inserting $u_{0}(x) + \sum_{n=1}^{2}\zeta_n(x)\,v_{0,\,n}(\eps^{-1}(x-M_n))$ into (\ref{PbChampTotalBIS}), we get the discrepancy\footnote{We call ``discrepancy'' the error on the source term in $\Om^{\eps}$.} $\sum_{n=1}^{2}(\Delta_x+k^2\mrm{Id})\,(\zeta_n(x)\,v_{0,\,n}(\eps^{-1}(x-M_n)))$. Remark that this discrepancy is defined only in $\Om^{\eps}$ and not in $\Om^{0}$. However, using (\ref{FarFieldOfNearField}), we replace it by its main contribution at infinity and we choose $u_{1}$ as the solution to the problem 
\begin{equation}\label{PbChampTotalTerm1}
\begin{array}{|rcll}
\multicolumn{4}{|l}{\mbox{Find }u_{1}\in\mrm{H}^1_{\loc}(\Om^{0}) \mbox{ such that }u_{1}\mbox{ is outgoing and } }\\[3pt]
-\Delta u_{1}-k^2 u_{1}  & = & -\dsp\sum_{n=1}^{2} \left([\Delta_x,\zeta_n]+k^2\zeta_n\mrm{Id}\right)\left(w^+(M_n)\,\frac{\Capa(\mathcal{O})}{|x-M_n|}\right) & \mbox{ in }\Om^{0}\\[12pt]
 u_{1}  & = & 0  & \mbox{ on }\Gamma^{0}.
\end{array}
\end{equation}
In (\ref{PbChampTotalTerm1}), $[\Delta_x,\zeta_n]$ denotes the commutator such that $[\Delta_x,\zeta_n]\varphi:=\Delta_x(\zeta_n \varphi)-\zeta_n\Delta_x \varphi=2\nabla\varphi\cdot\nabla \zeta_n+ \varphi\Delta \zeta_n$. Since $\zeta_n$ is equal to one in a neighbourhood of $M_n$ and compactly supported, note that the right hand side of (\ref{PbChampTotalTerm1}) is an element of $\mrm{L}^2(\Om^0)$. As a consequence, by elliptic regularity results (see e.g. \cite{LiMa68}), $u_1$ belongs to $\mH^2(K)$ for all bounded domains $K\subset\Om^0$. A Taylor expansion at $M_n$ gives, for $x\in\partial \mathcal{O}^{\eps}_n$,  
\[
\begin{array}{lcl}
(u_{0} + \dsp\sum_{n=1}^{2}\zeta_n(x)\,v_{0,\,n}(\eps^{-1}(x-M_n)) +\eps u_1)(x) & = & \eps(u_1(M_n)+(\eps^{-1}(x-M_n))\cdot\nabla u_0(M_n))+\dots\, \\[5pt] 
& = & \eps(u_1(M_n)+(\eps^{-1}(x-M_n))\cdot\nabla w^+(M_n))+\dots\,. 
\end{array}
\]
In order to satisfy the Dirichlet boundary condition on the obstacles at order $\eps$, we find that $v_{1,\,n}$ must verify
\begin{equation}\label{defSdnTerm1}
v_{1,\,n}(\xi_n) = -(u_1(M_n)\,W(\xi_n)+\nabla w^+(M_n)\cdot\overrightarrow{W}(\xi_n)).
\end{equation}
In (\ref{defSdnTerm1}), the vector valued function $\overrightarrow{W}=(W_1,W_2,W_3)^{\top}$ is such that, for $j=1,2,3$, $W_j$ is harmonic in $\R^3\setminus\overline{\mathcal{O}}$, vanishes at infinity and verifies $W_{j}=\xi_{j}$ on $\partial\mathcal{O}$ (we use the notaton $\xi=(\xi_1,\xi_2,\xi_3)^{\top}$). At $|\xi|\to+\infty$ , the asymptotic behaviours of $W$ and $\overrightarrow{W}$ are related via the formula (see e.g. \cite[\S6.4.3]{NaPl94})
\[
\left(\begin{array}{c}
W(\xi) \\[4pt] \overrightarrow{W}(\xi) 
\end{array}\right)= \mathscr{M}\left(\begin{array}{c}
\Phi(\xi) \\[4pt] \nabla \Phi(\xi)
\end{array}\right)
+O(|\xi|^{-3}),\qquad\mbox{where }\mathscr{M}:=\left(\begin{array}{cc}
-4\pi\,\Capa(\mathcal{O}) & \vec{p}^{\top}\\[4pt]
\vec{p} & P 
\end{array}\right).
\]
We remind the reader that $\Phi=\xi\mapsto-1/(4\pi|\xi|)$ is the fundamental solution of the Laplace operator in $\R^3$. The matrix $\mathscr{M}$ is called the \textit{polarization tensor} \cite[Appendix G]{PoSz51}. It is symmetric and with our special choice of the origin of the system of coordinates (see (\ref{CalculusChangeOfCoor})), the vector $\vec{p}$ is equal to zero. Therefore, for $j=1,2,3$, we have  
\begin{equation}\label{AsymptoticBehaviour1}
W_{j}(\xi) =   O(|\xi|^{-2}),\qquad \mbox{when } |\xi|\to+\infty.
\end{equation}
$\star$ Finally, we consider the expansion of $u^{\eps}$ at order $\eps^2$. After inserting $ u_{0}(x) + \sum_{n=1}^{2}\zeta_n(x)\,v_{0,\,n}(\eps^{-1}(x-M_n)) + \eps ( u_{1} + \sum_{n=1}^{2}\zeta_n(x)\,v_{1,\,n}(\eps^{-1}(x-M_n)))$ into (\ref{PbChampTotalBIS}), using formulas (\ref{FarFieldOfNearField}), (\ref{defSdnTerm0}) and (\ref{AsymptoticBehaviour1}), we get the discrepancy
\[
\dsp\sum_{n=1}^{2} \left([\Delta_x,\zeta_n]+k^2\zeta_n\mrm{Id}\right)\left(u_{1}(M_n)\,\frac{\Capa(\mathcal{O})}{|x-M_n|}\right) \eps^2+O(\eps^3).
\]
This leads us to define $u_2$ as the solution to the problem
\begin{equation}\label{PbChampTotalu2}
\begin{array}{|rcll}
\multicolumn{4}{|l}{\mbox{Find }u_{2}\in\mrm{H}^1_{\loc}(\Om^{0}) \mbox{ such that }u_{2}\mbox{ is outgoing and } }\\[3pt]
-\Delta u_{2}-k^2 u_{2}  & = & -\dsp\sum_{n=1}^{2} \left([\Delta_x,\zeta_n]+k^2\zeta_n\mrm{Id}\right)\left(u_{1}(M_n)\,\frac{\Capa(\mathcal{O})}{|x-M_n|}\right) & \mbox{ in }\Om^{0}\\[12pt]
 u_{2}  & = & 0  & \mbox{ on }\Gamma^{0}.
\end{array}
\end{equation}
The derivation of the problems satisfied by the terms $v_{2,\,n}$, $n=1,2$, will not be needed in the rest of the analysis. Therefore we do not describe it. In the next section, from the asymptotic expansion of $u^{\eps}$, solution to (\ref{PbChampTotalBIS}), we deduce an asymptotic expansion of the coefficients $s^{\eps\pm}$ appearing in the decomposition (\ref{DecompoChampScattered}).

\section{Asymptotic expansion of the transmission/reflection coefficients}\label{sectionAsymptoCoeff}

In accordance with the asymptotic expansion of $u^{\eps}$ (\ref{DefAnsatz}), for $s^{\eps\pm}$, we consider the ansatz
\begin{equation}\label{expansionReflecCoeff}
s^{\eps\pm} = s^{0\pm} + \eps\,s^{1\pm}+\eps^{2}\,s^{2\pm}+\dots\ .
\end{equation}
Now, we wish to compute $s^{0\pm}$, $s^{1\pm}$, $s^{2\pm}$. First, we give an explicit formula for $s^{\eps\pm}$. For $\pm z>\ell$ (we remind the reader that $\ell>0$ is chosen so that $\Om^{\eps}$ coincides with $\Om^{0}$ for $x=(y,z)$ such that $|z|\ge \ell/2$), using Fourier series we can decompose $u^{\eps}-w^+$ as 
\[
(u^{\eps}-w^+)(y,z) = s^{\eps\pm}w^{\pm}(y,z)+\sum_{n=2}^{+\infty} \alpha^{\pm}_n w^{\pm}_n(y,z)
\]
where $\alpha^{\pm}_n$ are some constants and where $w^{\pm}_n$ are defined in (\ref{defModes}). Set $\Sigma^{\ell}:=\left(\om\times\{-\ell\}\right)\cup\left(\om\times\{\ell\}\right)$. A direct calculation using the orthonormality of the family $(\varphi_n)_{n\ge1}$ and the expression of the $w^{\pm}$ (see (\ref{defModes})) yields the formulas
\begin{equation}\label{expressionReflecCoeff}
 i\,s^{\eps\pm} = \dsp\int_{\Sigma^{\ell}} \frac{\partial (u^{\eps}-w^+)}{\partial \nu}\,\overline{w^{\pm}}-(u^{\eps}-w^+)\frac{\partial\overline{w^{\pm}}}{\partial\nu}\,d\sigma,
\end{equation}
where $\partial_{\nu}=\pm \partial_{z}$ at $z=\pm \ell$.\\
\newline 
Observe that the correction terms $\zeta_n v_{k,\,n}$ appearing in the expansion of $u^{\eps}$ (\ref{DefAnsatz}) are compactly supported. As a consequence, they do not influence the coefficients $s^{\eps\pm}$. In particular, this implies $s^{0\pm}=0$. To compute $s^{1\pm}$, we plug (\ref{DefAnsatz}) in (\ref{expressionReflecCoeff}) and identify with (\ref{expansionReflecCoeff}) the powers in $\eps$. This gives
\[
 i\,s^{1\pm} = \dsp\int_{\Sigma^{\ell}} \frac{\partial u_{1}}{\partial\nu}\,\overline{w^{\pm}}- u_{1}\frac{\partial\overline{w^{\pm}}}{\partial\nu}\,d\sigma.
\]
Integrating by parts in $\om\times(-\ell;\,\ell)$ and using the equation $\Delta w^{\pm}+k^2w^{\pm}=0$ as well as (\ref{PbChampTotalTerm1}), we find 
\begin{equation}\label{IntegrationByPart1}
\begin{array}{lcl}
 i\,s^{1\pm} & = & \dsp\,\int_{\om\times(-\ell;\,\ell) }\Delta u_{1} \,\overline{w^{\pm}} - u_{1} \,\Delta\overline{w^{\pm}}\,dx\\[12pt]
& = & \dsp\int_{\Om^0}(\Delta u_{1}+k^2 u_{1}) \,\overline{w^{\pm}}\,dx\\[12pt]
& = & \dsp\sum_{n=1}^{2} \int_{\Om^0}([\Delta,\zeta_n]+k^2\zeta_n\mrm{Id})\left(u_{0}(M_n)\,\frac{\Capa(\mathcal{O})}{|x-M_n|}\right)\,\overline{w^{\pm}}\,dx\\[12pt]
& = & \dsp\sum_{n=1}^{2} u_{0}(M_n)\,\Capa(\mathcal{O})\int_{\Om^0}\overline{w^{\pm}}\,[\Delta,\zeta_n]\left(\frac{1}{|x-M_n|}\right)\,-\frac{\zeta_n\Delta \overline{w^{\pm}}}{|x-M_n|}\,dx.
\end{array}
\end{equation}
Denote $\mrm{B}^{\delta}_d(M_n)$ the ball of $\R^d$ centered at $M_n$ of radius $\delta$. 
Noticing that $[\Delta,\zeta_n](|x-M_n|^{-1})$ vanishes in a neighbourhood of $M_n$, $n=1,2$ (see the discussion after (\ref{PbChampTotalTerm1})) and using the estimate
\[
\left|\int_{\mrm{B}^{\delta}_3(M_n)}\frac{\Delta \overline{w^{\pm}}}{|x-M_n|} \,dx\,\right| = k^2 \left|\int_{\mrm{B}^{\delta}_3(M_n)}\frac{ \overline{w^{\pm}}}{|x-M_n|} \,dx\,\right| \le  C\,k^2\,\delta^2\,\|w^{\pm}\|_{\mrm{L}^{\infty}(\mrm{B}^{\delta}_3(M_n))},
\]
$C>0$ being a constant independent of $\delta$, we deduce from the last line of (\ref{IntegrationByPart1}) that
\begin{equation}\label{IntegrationByPartAutre}
\begin{array}{lcl}
i\,s^{1\pm} & = & \dsp\lim_{\delta\to 0}\ \sum_{n=1}^{2} u_{0}(M_n)\,\Capa(\mathcal{O})\int_{\Om^{0\delta}}\overline{w^{\pm}}\,[\Delta,\zeta_n]\left(\frac{1}{|x-M_n|}\right)\,-\frac{\zeta_n\Delta \overline{w^{\pm}}}{|x-M_n|}\,dx.
\end{array}
\end{equation}
In (\ref{IntegrationByPartAutre}), the set $\Om^{0\delta}$ is defined by  $\Om^{0\delta}:=\Om^0\setminus\dsp\cup_{n=1,2}\,\overline{\mrm{B}^{\delta}_3(M_n)}$. Remark that $[\Delta,\zeta_n](|x-M_n|^{-1})=\Delta(\zeta_n|x-M_n|^{-1})$ in $\Om^{0\delta}$. Using again that $\zeta_n$ is equal to one in a neighbourhood of $M_n$ and integrating by parts in (\ref{IntegrationByPartAutre}), we get
\begin{equation}\label{explicitCalculus}
i\,s^{1\pm}  =  \dsp\lim_{\delta\to 0} \ \sum_{n=1}^{2} u_{0}(M_n)\,\Capa(\mathcal{O}) \dsp\int_{\partial\mrm{B}^{\delta}_3(M_n)} \overline{w^{\pm}}\,\partial_{\nu}(|x-M_n|^{-1})-|x-M_n|^{-1}\partial_{\nu}\overline{w^{\pm}}\,d\sigma.
\end{equation}
In this expression, $\nu$ stands for the normal unit vector to $\partial\mrm{B}^{\delta}_3(M_n)$ directed to the interior of $\mrm{B}^{\delta}_3(M_n)$. Then, a direct computation using the relations $u_{0}(M_n)=w^+(M_n)=\overline{w^-(M_n)}$ gives 
\begin{equation}\label{ExpressionCoef1}
s^{1 +} = 4 i \pi \,\Capa(\mathcal{O})\dsp\sum_{n=1}^2|w^+(M_n)|^2\qquad\mbox{ and }\qquad
s^{1 -} = 4 i \pi \,\Capa(\mathcal{O})\dsp\sum_{n=1}^2w^+(M_n)^2.
\end{equation}
Working analogously from the formulas
\[
i\,s^{2\pm} = \dsp\int_{\Sigma^{\ell}} \frac{\partial u_{2}}{\partial\nu}\,\overline{w^{\pm}} - u_{2}\frac{\partial\overline{w^{\pm}}}{\partial\nu}\,d\sigma,
\]
we obtain
\begin{equation}\label{ExpressionCoef2}
s^{2+} = 4 i \pi \,\Capa(\mathcal{O})\dsp\sum_{n=1}^2 u_{1}(M_n)w^-(M_n)\quad\mbox{ and }\quad
s^{2-} = 4 i \pi \,\Capa(\mathcal{O})\dsp\sum_{n=1}^2 u_{1}(M_n)w^+(M_n).
\end{equation}

\section{Perturbation of the position of one fly}\label{sectionPerturbation}

In this section, we explain how to choose the position of the flies so that the reflection coefficient $s^{\eps-}$ in the decomposition of $u^{\eps}-w^+$ vanishes. In the previous analysis, we obtained the formula
\begin{equation}\label{AsymptoticExpan}
(4i\pi \,\Capa(\mathcal{O}))^{-1}s^{\eps-}= 0+ \eps\dsp\sum_{n=1}^2w^+(M_n)^2+\eps^2\dsp\sum_{n=1}^2 u_{1}(M_n)w^+(M_n) +O(\eps^3).
\end{equation}
First observe, that it is easy to cancel the term of order $\eps$ in the expansion (\ref{AsymptoticExpan}). Indeed, remembering that $w^+(x) = (2\beta_1)^{-1/2} e^{ i \beta_1 z}\varphi_1(y)$, we obtain 
\begin{equation}\label{relationToSatisfy} 
\dsp\sum_{n=1}^2w^+(M_n)^2 =0 \qquad\Leftrightarrow \qquad e^{2i \beta_1 z_1}\varphi_1(y_1)^{2}+e^{2i \beta_1 z_2}\varphi_1(y_2)^{2}=0.
\end{equation}
In order to satisfy (\ref{relationToSatisfy}), for example we choose $y_1, y_2\in\om$ such that $\varphi_1(y_1)=\varphi_1(y_2)$. Then we take $z_1=0$ and $z_2=(2m+1)\pi/(2\beta_1)$ for some $m\in\N=\{0,1,2,\dots\}$.\\
\newline
However, this is not sufficient since we want to impose  $s^{\eps-}=0$ at all orders in $\eps$. The terms of orders $\eps^2$, $\eps^3\dots$ in (\ref{AsymptoticExpan}) have a less explicit dependence with respect to $M_1$, $M_2$. In particular, this dependence is non linear. Therefore, it is not obvious that we can find an explicit formula for the positions of the flies ensuring $s^{\eps-}=0$. To cope with this problem, we will introduce new degrees of freedom slightly changing the position of one fly. To set ideas, we assume that the fly situated at $M_1$ moves from $M_1$ to $M_1^{\tau}=M_1+\eps\tau$. Here, $\tau$ is an element of $\R^3$ to determine which offers a priori three degrees of freedom. In the following, for a given $\eps>0$, we show how $\tau$ can be chosen as the solution of a fixed point problem to cancel the reflection coefficient.\\
\newline
We define $\mathcal{O}^\eps_1(\tau) := \{x\in\R^3\,|\,\eps^{-1}(x-M^{\tau}_1)\in\mathcal{O}\}$ and $\Om^{\eps}(\tau) :=  \Om^{0}\setminus(\overline{\mathcal{O}^\eps_1(\tau)}\cup\overline{\mathcal{O}^\eps_2})$. Problem 
\begin{equation}\label{PbChampTotalPerturb}
\begin{array}{|rcll}
\multicolumn{4}{|l}{\mbox{Find }u^{\eps}(\tau)\in\mrm{H}^1_{\loc}(\Om^{\eps}(\tau)) \mbox{ such that }u^{\eps}(\tau)-w^+\mbox{ is outgoing and } }\\[3pt]
-\Delta u^{\eps}(\tau)  & = & k^2 u^{\eps}(\tau) & \mbox{ in }\Om^{\eps}(\tau)\\[3pt]
 u^{\eps}(\tau)  & = & 0  & \mbox{ on }\Gamma^{\eps}(\tau):=\partial\Om^{\eps}(\tau),
\end{array}
\end{equation}
has a unique solution $u^{\eps}(\tau)$ for $\eps$ small enough. It admits the decomposition
\begin{equation}\label{DecompoChampTotalBis}
u^{\eps}(\tau)-w^+ = \chi^+ s^{\eps+}(\tau) w^+ + \chi^- s^{\eps-}(\tau) w^- + \tilde{u}^{\eps}(\tau),
\end{equation}
where $s^{\eps\pm}(\tau)\in\Cplx$ and where $\tilde{u}^{\eps}(\tau)\in\mH^1(\Om^{\eps}(\tau))$ is a term which is exponentially decaying at $\pm\infty$. From the previous section (see formulas (\ref{ExpressionCoef1}), (\ref{ExpressionCoef2})), we know that 
\[
(4i\pi \,\Capa(\mathcal{O}))^{-1}s^{\eps-}(0) = 0+ \eps\dsp\sum_{n=1}^2w^+(M_n)^2+\eps^2\dsp\sum_{n=1}^2 u_{1}(M_n)w^+(M_n) +O(\eps^3).
\]
Formally, we also have (we will justify this expansion in Section \ref{JustiAsymptotics})
\begin{equation}\label{dlWithTau}
\begin{array}{lcl}
(4i\pi \,\Capa(\mathcal{O}))^{-1}s^{\eps-}(\tau) & = & 0+ \eps\,(w^+(M_1^{\tau})^2+w^+(M_2)^2)\\[4pt]
 & & \phantom{0 }+\eps^2\,(u_{1}(M_1^{\tau})w^+(M_1^{\tau})+u_{1}(M_2)w^+(M_2)) +O(\eps^3).
\end{array}
\end{equation}
To remove the $\eps$ dependence hidden in the term $w^+(M^{\tau}_1)$, we use the Taylor expansion  
\[
w^+(M^{\tau}_1) = w^+(M_1+\eps\tau) = w^+(M_1)+\eps\tau\cdot\nabla w^+(M_1) + O(\eps^2).
\]
This gives
\begin{equation}\label{relation1}
w^+(M^{\tau}_1)^2 = w^+(M_1)^2+2w^+(M_1)\,\tau\cdot\nabla w^+(M_1)\,\eps + O(\eps^2).
\end{equation}
Plugging (\ref{relation1}) and using the relation $u_{1}(M_1^{\tau})w^+(M_1^{\tau})=u_{1}(M_1)w^+(M_1)+O(\eps)$ in (\ref{dlWithTau}) lead to 
\begin{equation}\label{DvpCoefRef}
\begin{array}{lcl}
(4i\pi \,\Capa(\mathcal{O}))^{-1}s^{\eps-}(\tau) &\hspace{-0.25cm} = & \hspace{-0.25cm}0+\eps\dsp\sum_{n=1}^2w^+(M_n)^2\\[10pt]
&\hspace{-0.25cm} & \hspace{-0.25cm}\dsp\phantom{0 }+\eps^2\,\Big(2 w^+(M_1)\,\tau\cdot\nabla w^+(M_1)+\sum_{n=1}^2 u_{1}(M_n)w^+(M_n)\Big)+\eps^3\,\tilde{s}^{\eps-}(\tau).
\end{array}
\end{equation}
In (\ref{DvpCoefRef}), $\tilde{s}^{\eps-}(\tau)$ denotes an abstract remainder (see \S\ref{paragraphDefiRemainder} for its definition). With this new parameterization, for a given $\eps$, our goal is to find $\tau\in\R^3$ such that $s^{\eps-}(\tau)=0$. Set $\tau=(\tau_y,\tau_z)^{\top}$ with $\tau_y\in\R^2$, $\tau_z\in\R$. Since $w^{+}(y,z)= (2\beta_1)^{-1/2} e^{i \beta_1 z}\varphi_1(y)$, we find 
\begin{equation}\label{explicitCalculus1}
2w^+(M_1)\,\tau\cdot\nabla w^+(M_1) = \varphi_1(y_1)\,\beta_1^{-1}\left(\tau_y\cdot\nabla_y\varphi_1(y_1)+i\beta_1\tau_z\varphi_1(y_1)\right).
\end{equation}
Now, choose $y_1$ such that $\nabla_y\varphi_1(y_1)\ne0$. Remark that it is possible since for the moment, we have only imposed that $y_1$, $y_2$ are such that $\varphi_1(y_1)=\varphi_1(y_2)$. Let us look for $\tau$ under the form
\begin{equation}\label{DecompoTau}
\begin{array}{ll}
\tau = & \dsp\phantom{+ }\varphi_1(y_1)^{-1}\,\beta_1\,\Big(\kappa_y-\Re e\,\sum_{n=1}^2 u_{1}(M_n)w^+(M_n)\Big) (\nabla_y\varphi_1(y_1)/|\nabla_y\varphi_1(y_1)|^2,0 )^{\top}\\[6pt]
  & \dsp+ 
\varphi_1(y_1)^{-1}\,\beta_1\,
\Big(\kappa_z-\Im m\,\sum_{n=1}^2 u_{1}(M_n)w^+(M_n)\Big)(0,0,(\beta_1\,\varphi_1(y_1))^{-1})^{\top},
\end{array}
\end{equation}
where $\kappa_y$, $\kappa_z$ are some real parameters to determine. Here, we emphasize that it is sufficient to parameterize the unknown $\tau\in\R^3$ with only two real degrees of freedom (and not three) because we want to cancel one complex coefficient. To impose $s^{\eps-}(\tau)=0$, plugging (\ref{DecompoTau}) in (\ref{DvpCoefRef}), we see that we have to solve the problem 
\begin{equation}\label{PbFixedPoint}
\begin{array}{|l}
\mbox{Find }\kappa=(\kappa_y,\kappa_z)\in\R^{2}\mbox{ such that }\kappa = \mathscr{F}^{\eps}(\kappa),
\end{array}~\\[5pt]
\end{equation}
with 
\begin{equation}\label{DefMapContra}
\mathscr{F}^{\eps}(\kappa):= -\eps\,(\Re e\,\tilde{s}^{\eps-}(\tau),\,\Im m\,\tilde{s}^{\eps-}(\tau))^{\top}. 
\end{equation}
Proposition \ref{propoTech} hereafter ensures that for any given parameter $\gamma>0$, there is some $\eps_0>0$ such that for all $\eps\in(0;\eps_0]$, the map $\mathscr{F}^{\eps}$ is a contraction of $\mrm{B}^{\gamma}_2(O):=\{\kappa\in\R^{2}\,\big|\,|\kappa| \le \gamma\}$. Therefore, the Banach fixed-point theorem guarantees the existence of some $\eps_0>0$ such that for all $\eps\in(0;\eps_0]$, Problem (\ref{PbFixedPoint}) has a unique solution in $\mrm{B}^{\gamma}_2(O)$. Note that for a given $\gamma>0$, $\eps_0>0$ can be chosen small enough so that there holds $\overline{\mathcal{O}^\eps_1(\tau)}\subset\Om^0$ (the first fly stays in the reference waveguide). From the previous analysis we deduce the following proposition.
\begin{proposition}\label{propositionMainResult}
There exists $\eps_0>0$ such that for all $\eps\in(0;\eps_0]$, we can find $M_1^{\tau}=M_1+\eps\tau$, $M_2\in\Om^0$ such that the reflection coefficient $s^{\eps-}(\tau)$ in the decomposition (\ref{DecompoChampTotalBis}) of the function $u^{\eps}(\tau)$ vanishes. More precisely, we can take $M_1=(y_1,z_1)$, $M_2=(y_2,z_2)$ satisfying
\begin{equation}\label{assumptionPos}
\hspace{-6.5cm}\begin{array}{l}
\mbox{$\bullet\ $ $z_1=0$, $z_2=(2m+1)\pi/(2\beta_1)$, $m\in\N$;}\\[6pt]
\mbox{$\bullet\ $ $y_1, y_2\in\om$ such that $\varphi_1(y_1)=\varphi_1(y_2)$, $\nabla_y\varphi_1(y_1)\ne0$;}
\end{array}
\end{equation}
and $\tau$ defined by (\ref{DecompoTau}), where $(\kappa_y,\kappa_z)$ is a solution of the fixed point problem (\ref{DefMapContra}).
\end{proposition}

\begin{remark}\label{RmqSwarm}
Note that it is necessary to play with at least two small flies to cancel the reflection coefficient $s^{\eps-}$. Indeed, if there is only one fly, the first term $4i\pi \,\Capa(\mathcal{O})w^+(M_1)^2$ in the asymptotic expansion (\ref{AsymptoticExpan}) of $s^{\eps-}$ cannot be made equal to zero playing with the position $M_1$. Therefore, the incoming wave $w^+$ generates a scattered field whose amplitude is of order $\eps$. The flies really have to act as partners to become invisible.
\end{remark}

\begin{remark}
Assume that the waveguide $\Om^0$ contains $N$ flies, located at $M_1,\dots, M_N$ ($N$ distinct points of $\Om^0$), which coincide with the sets $\mathcal{O}^\eps_n = \{x\in\R^3\,|\,\eps^{-1}(x-M_n)\in\mathcal{O}\}$, $n=1,\dots,N$. For $n=1,\dots, N$, introduce $y_n\in\om$, $z_n\in\R$ such that $M_n=(y_n,z_n)$. We can hide these $N$ flies positioning cleverly $N$ other flies in $\Om^0$. Indeed, for $n=1,\dots, N$, define $M_{N+n}:=(y_n,z_n+(2m+1)\pi/(2\beta_1))$ and $\mathcal{O}^\eps_{N+n} = \{x\in\R^3\,|\,\eps^{-1}(x-M_{N+n})\in\mathcal{O}\}$. Here $m\in \N$ is set so that the points $M_1,\dots, M_{2N}$ are all distinct. With this choice, we have  
\[
\dsp\sum_{n=1}^{2N}w^+(M_n)^2 = (2\beta_1)^{-1}\dsp\sum_{n=1}^{N}(e^{2i \beta_1 z_n}\varphi_1(y_n)^{2}+e^{ 2i\beta_1 z_n+(2m_n+1)i\pi}\varphi_1(y_n)^{2}) =0.
\]
Therefore, we can cancel the term of order $\eps$ in the asymptotic expansion of $s^{\eps-}$ (see (\ref{AsymptoticExpan})). Then, if there is one fly which is not located at an extremum of the first eigenfunction $\varphi_1$ (so that $\nabla_y\varphi_1(y)\ne 0$), we can proceed as above slightly perturbing the position of this fly to achieve $s^{\eps-}=0$. 
\end{remark}

\noindent Before proceeding further, we show that the map $\mathscr{F}^{\eps}$ is a contraction as required by the analysis preceding Proposition \ref{propositionMainResult}.
\begin{proposition}\label{propoTech}
Let $\gamma>0$ be a given parameter. Then, there exists $\eps_0>0$ such that for all $\eps\in(0;\eps_0]$, the map $\mathscr{F}^{\eps}$ defined by (\ref{DefMapContra}) is a contraction of $\mrm{B}^{\gamma}_2(O)=\{\kappa\in\R^{2}\,\big|\,|\kappa| \le \gamma\}$.
\end{proposition}
\begin{proof}
For $\kappa=(\kappa_y,\kappa_z)\in\R^{2}$, we have $\mathscr{F}^{\eps}(\kappa)= -\eps\,(\Re e\,\tilde{s}^{\eps-}(\tau),\,\Im m\,\tilde{s}^{\eps-}(\tau))^{\top}$ where $\tau$ is defined from $\kappa$ according to (\ref{DecompoTau}) and where $\tilde{s}^{\eps-}(\tau)$ is the remainder appearing in (\ref{DvpCoefRef}). In Section \ref{JustiAsymptotics}, we will prove that for all $\vartheta>0$, there is some $\eps_0>0$ such that for all $\eps\in(0;\eps_0]$, there holds 
\begin{equation}\label{MainEstimate}
|\tilde{s}^{\eps-}(\tau)-\tilde{s}^{\eps-}(\tau')| \le C\,|\tau-\tau'|,\qquad \forall\tau,\,\tau'\in \mrm{B}^{\vartheta}_3(O)=\{\tau\in\R^{3}\,\big|\,|\tau| \le \vartheta\}.
\end{equation}
Here and in what follows, $C>0$ is a constant which may change from one occurrence to another but which is independent of $\eps$ and $\kappa,\kappa'\in\mrm{B}^{\gamma}_2(O)$. Since $\tau$ belongs to a bounded set of $\R^3$ when $\kappa\in \mrm{B}^{\gamma}_2(O)$, (\ref{MainEstimate}) yields
\begin{equation}\label{EstimContra}
|\mathscr{F}^{\eps}(\kappa)-\mathscr{F}^{\eps}(\kappa')|\le C\,\eps\,|\kappa-\kappa'|,\qquad \forall\kappa,\kappa'\in\mrm{B}^{\gamma}_2(O).
\end{equation}
Taking $\kappa'=0$ in (\ref{EstimContra}) and remarking that $|\mathscr{F}^{\eps}(0)| \le C\,\eps$, we find $|\mathscr{F}^{\eps}(\kappa)| \le C\,\eps$ for all $\kappa\in\mrm{B}^{\gamma}_2(O)$. With (\ref{EstimContra}), this allows us to conclude that the map $\mathscr{F}^{\eps}$ is a contraction of $\mrm{B}^{\gamma}_2(O)$ for $\eps$ small enough.
\end{proof}
\begin{remark}
\noindent Let us denote $\kappa_{\mrm{sol}}\in\mrm{B}^{\gamma}_2(O)$ the unique solution to Problem (\ref{PbFixedPoint}). The previous proof ensures that there exists a constant $c_0>0$ independent of $\eps$ such that 
\begin{equation}\label{RelAPos}
|\kappa_{\mrm{sol}}| = |\mathscr{F}^{\eps}(\kappa_{\mrm{sol}})| \le c_0\,\eps,\qquad \forall\eps\in(0;\eps_0].
\end{equation}
Introduce $\tau_{\mrm{sol}},\tau_0$ the vectors of $\R^3$ respectively defined from $\kappa_{\mrm{sol}}$, $\kappa=0$ using formula (\ref{DecompoTau}). In particular, we have
\begin{equation}\label{DecompoTau0}
\begin{array}{ll}
\tau_0 = & \dsp- \varphi_1(y_1)^{-1}\,\beta_1\,\Re e\,\sum_{n=1}^2 u_{1}(M_n)w^+(M_n)\,(\nabla_y\varphi_1(y_1)/|\nabla_y\varphi_1(y_1)|^2,0)^{\top}\\[6pt]
& \dsp-\varphi_1(y_1)^{-1}\,\beta_1\,
\Im m\,\sum_{n=1}^2 u_{1}(M_n)w^+(M_n)\,(0,0,(\beta_1\,\varphi_1(y_1))^{-1})^{\top}.
\end{array}
\end{equation}
We know (see (\ref{DecompoTau})) that $|M_1^{\tau_{\mrm{sol}}}-M_1^{\tau_{0}}| = \eps |\tau_{\mrm{sol}}-\tau_{0}| \le C\,\eps\,|\kappa_{\mrm{sol}}|$. From estimate (\ref{RelAPos}), we deduce that $|M_1^{\tau_{\mrm{sol}}}-M_1^{\tau_{0}}|  \le C\,\eps^2$. As a consequence, we can say that $M_1^{\tau_{\mrm{sol}}}$ is equal to $M_1^{\tau_{0}}=M_1+\eps\tau_{0}$ at order $\eps$. 
\end{remark}

\section{Obstruction to transmission invisibility}\label{sectionTransInv}
Let us consider again $u^{\eps}$ the solution to Problem (\ref{PbChampTotalBIS}). We have introduced the coefficients $ s^{\eps\pm}$ such that the scattered field $u^{\eps}-w^+$  admits the expansion
\begin{equation}\label{DecompoChampScatteredRe}
u^{\eps}-w^+ = \chi^+ s^{\eps+} w^+ + \chi^- s^{\eps-} w^- + \tilde{u}^{\eps},
\end{equation}
where $\tilde{u}^{\eps}\in\mH^1(\Om^{\eps})$ is a term exponentially decaying at $\pm\infty$. With this notation, the usual reflection and transmission coefficients are respectively given by $R^{\eps}=s^{\eps-}$ and $T^{\eps} = 1+s^{\eps+}$. Up to now, we have explained how to cancel $s^{\eps-}$ so that the flies are invisible to an observer sending the incident waves $w^+$ and measuring the resulting scattered field at $-\infty$. Now, assume that the observer can also measure the scattered field at $+\infty$. Can we hide the flies in this setting? Equivalently, can we impose $s^{\eps+}=0$ (or $T^{\eps} = 1$) so that the scattered field is also exponentially decaying at $+\infty$? First, remark that the approach of Section \ref{sectionPerturbation} can not be implemented to impose $s^{\eps+}=0$. And more generally, it is easy to see that the latter relation cannot be obtained when there are small flies in the reference waveguide. Indeed, from (\ref{expansionReflecCoeff}), (\ref{ExpressionCoef1}), we know that $s^{\eps+}$ admits the asymptotic expansion 
\[
s^{\eps+} =  4 i \pi\,\Capa(\mathcal{O})\,\dsp\sum_{n=1}^2|w^+(M_n)|^2\eps+O(\eps^2).
\]
Since $\Capa(\mathcal{O})>0$, we have $\Im m\,s^{\eps+}>0$ for $\eps$ small enough. Now, we prove another result showing that it is impossible to hide any (not necessarily small) sound soft obstacle for wavenumvers $k$ smaller than a constant that we explicit.\\

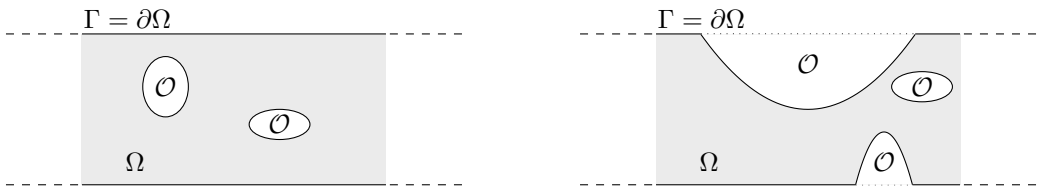
\begin{figure}[!ht]
\centering
\begin{tikzpicture}
\draw[fill=gray!16,draw=none](-2,0) rectangle (2,2);
\draw (-2,0)--(2,0);
\draw (-2,2)--(2,2);
\draw [dashed](-3,0)--(-2,0);
\draw [dashed](3,0)--(2,0);
\draw [dashed](-3,2)--(-2,2);
\draw [dashed](3,2)--(2,2);
\node at (-1.3,0.3){\small$\Omega$};
\node at (-1.4,2.2){\small$\Gamma=\partial\Om$};
\draw[fill=white] (0.6,0.8) ellipse (0.4 and 0.2);
\draw[fill=white] (-0.9,1.3) ellipse (0.3 and 0.4);
\node at (-0.9,1.3){\small$\mathcal{O}$};
\node at (0.6,0.8){\small$\mathcal{O}$};
\end{tikzpicture}\qquad\qquad\begin{tikzpicture}
\draw[fill=gray!16,draw=none](-2,0) rectangle (2,2);
\draw (-2,0)--(0.62583426132,0);
\draw (2,0)--(1.37416573868,0);
\draw [dotted](0.62583426132,0)--(1.37416573868,0);
\draw (-2,2)--(-1.41,2);
\draw (2,2)--(1.41,2);
\draw [dotted](-1.41,2)--(1.41,2);
\draw[fill=white,domain=-1.41:1.41,smooth,variable=\x] plot ({\x},{0.5*\x*\x+1});
\draw[fill=white,domain=0.62583426132:1.37416573868,smooth,variable=\x] plot ({\x},{-5*(\x-1)*(\x-1)+0.7});
\draw [dashed](-3,0)--(-2,0);
\draw [dashed](3,0)--(2,0);
\draw [dashed](-3,2)--(-2,2);
\draw [dashed](3,2)--(2,2);
\draw[fill=white] (1.5,1.3) ellipse (0.4 and 0.2);
\node at (1.5,1.3){\small$\mathcal{O}$};
\node at (-1.3,0.3){\small$\Omega$};
\node at (-1.4,2.2){\small$\Gamma=\partial\Omega$};
\node at (0,1.6){\small$\mathcal{O}$};
\node at (1,0.3){\small$\mathcal{O}$};
\end{tikzpicture}
\caption{Examples of perturbed waveguides for which transmission invisibility cannot be imposed for wavenumbers smaller than a $k_{\star}$. \label{DomainObstruction}} 
\end{figure}

\noindent Since the result that we will prove holds in a more general setting than the one described at the beginning of the paper, we need to modify a bit the notation. Let $\mathcal{O}$ be a bounded domain with Lipschitz boundary verifying $\mathcal{O}\subset\Om^0$. We define the perturbed waveguide $\Om:=\Om^0\setminus\overline{\mathcal{O}}$ and make the assumption that $\mathcal{O}$  is such that $\Om$ is connected with a Lipschitz boundary (see Figure \ref{DomainObstruction} for examples of geometries fulfilling these criteria). We consider the scattering problem 
\begin{equation}\label{PbChampTotalTER}
\begin{array}{|rcll}
\multicolumn{4}{|l}{\mbox{Find }u\in\mrm{H}^1_{\loc}(\Om) \mbox{ such that }u-w^+\mbox{ is outgoing and } }\\[3pt]
-\Delta u  & = & k^2 u & \mbox{ in }\Om\\[3pt]
u  & = & 0  & \mbox{ on }\Gamma:=\partial\Om.\end{array}
\end{equation}
Problem (\ref{PbChampTotalTER}) always admits a solution $u$ with the decomposition 
\begin{equation}\label{DecompoChampTotalObstruction}
u-\chi^-w^+ = \chi^+ T w^+ + \chi^- R w^- + \tilde{u},
\end{equation}
where $R,\,T\in\Cplx$ are uniquely defined and where $\tilde{u}\in\mH^1(\Om)$ is exponentially decaying at $\pm\infty$. In (\ref{DecompoChampTotalObstruction}), as in (\ref{scatteredField}), $\chi^{\pm}\in\mathscr{C}^{\infty}(\Om_0)$ are cut-off functions equal to one for $\pm z\ge \ell$ and equal to zero for $\pm z\le \ell/2$. Moreover, the constant $\ell>0$ is chosen large enough so that $\Om$ coincides with $\Om^{0}$ for $x=(y,z)$ verifying $|z|\ge \ell/2$. We wish to prove that for $k<k_{\star}$ with $k^2_{\star}\in(\lambda_1;\lambda_2]$, we cannot have $T=1$ so that perfect transmission invisibility cannot be imposed. We denote $u_i:=w^+$ and $u_s:=u-u_i$. We also introduce the coefficients $s^{\pm}$ such that $s^{-} = R$ and $s^{+} = T-1$. With this definition, according to (\ref{DecompoChampTotalObstruction}), we have $u_s= \chi^+ s^+ w^+ + \chi^- s^- w^- + \hat{u}$ for some $\hat{u}\in\mH^1(\Om)$ which is exponentially decaying at $\pm\infty$.

\begin{lemma}\label{lemmaInitial}
Assume that the transmission coefficient $T$ in the decomposition (\ref{DecompoChampTotalObstruction}) satisfies $T=1$. Then there holds
\begin{equation}\label{identityLocal}
\int_{\Om}|\nabla u_s|^2 - k^2|u_s|^2\,dx +\int_{\mathcal{O}}|\nabla u_i|^2 - k^2|u_i|^2\,dx=0.
\end{equation}
\end{lemma}
\begin{proof}
The conservation of energy ensures that
\begin{equation}\label{ConservationEnergy}
1=|T|^2+|R|^2 \Leftrightarrow 1=|1+s^+|^2+|s^-|^2 \Leftrightarrow |s^+|^2+|s^-|^2=-2\,\Re e\,s^+.
\end{equation}
Assume that $T$ satisfies $T=1\Leftrightarrow s^+=0$. In this case, according to (\ref{ConservationEnergy}), we must have $s^-=0$ and $u_s$ is exponentially decaying at $\pm\infty$. Therefore, the integrals appearing in (\ref{identityLocal}) are well-defined. Denote $\Om_\ell:=\{x=(y,z)\in\Om\,|\,-\ell<z<\ell \}$. From the equation $\Delta u_s+k^2 u_s=0$, multiplying by $\overline{u_s}$ and integrating by parts, we find
\begin{equation}\label{decompo0Local}
\int_{\Om_\ell}|\nabla u_s|^2 - k^2|u_s|^2\,dx-\int_{\partial\Om\cap\partial\mathcal{O}}\partial_{\nu} u_s \overline{u_s}\,d\sigma  -\int_{\Sigma^{\ell}}\partial_{\nu} u_s \overline{u_s}\,d\sigma= 0. 
\end{equation}
Here, we set $\Sigma^{\ell}=\left(\om\times\{-\ell\}\right)\cup\left(\om\times\{\ell\}\right)$ (as in (\ref{expressionReflecCoeff})) and $\nu$ stands for the normal unit vector to $\partial\Om_\ell$ directed to the exterior of $\Om_\ell$. On $\partial\Om$, we have $u=0$ which implies  $u_s=- u_i$. This allows us to write
\begin{equation}\label{decompo1Local}
\begin{array}{lcl}
-\dsp\int_{\partial\Om\cap\partial\mathcal{O}}\partial_{\nu} u_s \overline{u_s}\,d\sigma & = & \dsp\int_{\partial\Om\cap\partial\mathcal{O}}\partial_{\nu} u_s \overline{u_i}\,d\sigma\\[10pt]
& = & \dsp\int_{\partial\Om\cap\partial\mathcal{O}}\partial_{\nu} u_s \overline{u_i}\,d\sigma -\dsp\int_{\partial\Om\cap\partial\mathcal{O}} u_s \partial_{\nu}\overline{u_i}\,d\sigma-\dsp\int_{\partial\Om\cap\partial\mathcal{O}} u_i \partial_{\nu}\overline{u_i}\,d\sigma\\[10pt]
& = & -\dsp\int_{\Sigma^{\ell}}\partial_{\nu} u_s \overline{u_i}\,d\sigma +\dsp\int_{\Sigma^{\ell}} u_s \partial_{\nu}\overline{u_i}\,d\sigma-\dsp\int_{\partial\Om\cap\partial\mathcal{O}} u_i \partial_{\nu}\overline{u_i}\,d\sigma. 
\end{array}
\end{equation}
The third equality of (\ref{decompo1Local}) is a direct result of integrations by parts. Now, we consider the last term of the right hand side of (\ref{decompo1Local}). On $\partial\Om\cap\partial\mathcal{O}$, remark that $\nu$ is directed to the interior of $\mathcal{O}$. Therefore, integrating by parts, we obtain 
\begin{equation}\label{decompo2Local}
-\dsp\int_{\partial\Om\cap\partial\mathcal{O}} u_i \partial_{\nu}\overline{u_i}\,d\sigma=\dsp\int_{\mathcal{O}}|\nabla u_i|^2 - k^2|u_i|^2\,dx.
\end{equation}
Finally, gathering (\ref{decompo0Local}), (\ref{decompo1Local}), (\ref{decompo2Local}) and using formula (\ref{expressionReflecCoeff}), we obtain identity (\ref{identityLocal}).
\end{proof}
\noindent In the following, we use (\ref{identityLocal}) to show that we cannot have $T=1$ for $k<k_{\star}$ with $k^2_{\star}\in(\lambda_1;\lambda_2]$. 
\begin{proposition}\label{MainPropositionMixt}
Assume that $k^2<\min(\lambda_1+\pi^2/(2L)^2,\lambda_2)$ where $L>0$ denotes the smallest constant such that $\mathcal{O}\subset\mathcal{R}_{L}:=\om\times(-L;L)$. Then the transmission coefficient $T$ appearing in the decomposition (\ref{DecompoChampTotalObstruction}) satisfies $T\ne1$. 
\end{proposition}
\begin{remark}
Even though this result is stated in dimension $d=3$, it holds in any dimension $d\ge2$. And the proof for $d\ne3$ is the same as the one presented below. 
\end{remark}
\begin{remark}The authors do not know if the constant $\min(\lambda_1+\pi^2/(2L)^2,\lambda_2)$ is optimal.
\end{remark}
\begin{proof}
Assume, by contradiction, that $T=1$ so that (\ref{identityLocal}) holds. Let us write 
\begin{equation}\label{NullEnergySoft}
\int_{\Om}|\nabla u_s|^2 - k^2|u_s|^2\,dx=\int_{\Om\setminus\overline{\Om_L}}|\nabla u_s|^2 - k^2|u_s|^2\,dx+\int_{\Om_L}|\nabla u_s|^2 - k^2|u_s|^2\,dx.
\end{equation}
Now, we estimate each of the two terms of the right hand side of (\ref{NullEnergySoft}).\\
\newline
$\star$ Let us consider the first one. For $\pm z>L$, using Fourier series, classically, we can decompose $u_s$ as 
\[
u_s(y,z) = s^{\pm}w^{\pm}(y,z)+\sum_{n=2}^{+\infty} \alpha^{\pm}_n w^{\pm}_n(y,z).
\]
In this expression, according to (\ref{defModes}), we have 
\[
w^{\pm}(y,z) = \frac{1}{\sqrt{2\beta_1}}\,e^{\pm i \beta_1 z}\varphi_1(y)\qquad \mbox{ and, for $n\ge2$, }\qquad  w^{\pm}_n(y,z) = \frac{1}{\sqrt{2\beta_n}}\,e^{\mp \beta_n z}\varphi_n(y),
\]
where $\beta_1=\sqrt{k^2-\lambda_1}$ and, for $n\ge2$, $\beta_n=\sqrt{\lambda_n-k^2}$. Moreover, there holds 
\begin{equation}\label{defCoeff}
s^{\pm} = \sqrt{2\beta_1} e^{- ikL}\int_{\om\times\{\pm L\}} u_s\varphi_1\,d\sigma\qquad \mbox{ and, for $n\ge2$, }\qquad\alpha^{\pm}_n=\sqrt{2\beta_n} e^{+ \beta_nL}\int_{\om\times\{\pm L\}} u_s\varphi_n\,d\sigma.
\end{equation}
Define the regions $\Om_+=\om\times(L;+\infty)$ and $\Om_-=\om\times(-\infty;-L)$. A direct calculation using the orthonormality of the family $(\varphi_n)_{n\ge1}$ and the fact that $s^{\pm}=0$ when $T=1$ yields 
\begin{equation}\label{estim1}
\begin{array}{lcl}
\dsp\int_{\Om\setminus \overline{\Om_L}} |\nabla u_s|^2 - k^2|u_s|^2\,dx & = & \dsp\sum_{n=2}^{+\infty} 2(\lambda_n-k^2)\Big(|\alpha^{+}_n|^2 \dsp\int_{\Om_+}|w^{+}_n|^2\,dx+|\alpha^{-}_n|^2 \dsp\int_{\Om_-}|w^{-}_n|^2\,dx\Big)\\[10pt]
& \ge & 2(\lambda_2-k^2)\Big( \dsp\int_{\Om_+}|u_s|^2\,dx+\dsp\int_{\Om_-}|u_s|^2\,dx\Big). 
\end{array}
\end{equation}
$\star$ Now, we deal with the second term of the right hand side of (\ref{NullEnergySoft}). Let us extend $u_s$ to the domain $\mathcal{R}_{L}=\om\times(-L;L)$ introducing the function $\gamma$ such that 
\begin{equation}\label{gamma}
\gamma=\left\{ \begin{array}{lcl}
\phantom{-}u_s &  \mbox{in} & \Om_L \\
-u_i &  \mbox{in} & \mathcal{O}.
\end{array}\right.
\end{equation}
Since $u=u_i+u_s$ and $u=0$ on $\partial\Om\cap\partial\mathcal{O}$, clearly $\gamma$ belongs to $\mH^1(\mathcal{R}_{L})$. Above, we have seen that $s^{+}=0$ implies $s^{-}=0$. Then, according to (\ref{defCoeff}), there holds $\int_{\om\times\{+L\}}u_s\varphi_1\,d\sigma=\int_{\om\times\{-L\}}u_s\varphi_1\,d\sigma=0$. Define the Hilbert space $\mX:=\{\psi\in\mH^1(\mathcal{R}_{L})\,|\,\int_{\om\times\{+L\}}\psi\varphi_1\,d\sigma=\int_{\om\times\{-L\}}\psi\varphi_1\,d\sigma=0\mbox{ and }\psi=0\mbox{ on }\partial\Om^0\cap\partial\mathcal{R}_{L}\}$.
\begin{lemma}
We have the Poincar\'e inequality
\begin{equation}\label{PoincareMixt}
\int_{\mathcal{R}_{L}} |\psi|^2\,dx \le \frac{1}{\mu_1}\,\int_{\mathcal{R}_{L}} |\nabla\psi|^2\,dx,\qquad \forall \psi\in\mX,
\end{equation}
with $\mu_1:=\min(\lambda_1+\pi^2/(2L)^2,\lambda_2)$. 
\end{lemma}
\begin{proof}
Set $\mu_1=\inf_{\mX\setminus\{0\}}\int_{\mathcal{R}_{L}} |\nabla\psi|^2\,dx/\int_{\mathcal{R}_{L}} |\psi|^2\,dx$. One can check that $\mu_1$ coincides with the smallest eigenvalue of the problem
\begin{equation}\label{EigenPb}
\begin{array}{|rcll}
\multicolumn{4}{|l}{\mbox{Find }(\mu,\zeta)\in \R\times (\mX\setminus\{0\}) \mbox{ such that } }\\[4pt]
-\Delta \zeta  & = & \mu \zeta & \mbox{ in }\mathcal{R}_{L}\\[2pt]
\partial_{\nu} \zeta & = & c_+\,\varphi_1 & \mbox{ on }\om\times\{ +L\}\\[2pt]
\partial_{\nu} \zeta & = & c_-\,\varphi_1 & \mbox{ on }\om\times\{ -L\}
\end{array}
\end{equation}
where $c_{\pm}\in\R$ are some constants. Here, $\nu$ stands for the normal unit vector to $\partial \mathcal{R}_{L}$ directed to the exterior of $\mathcal{R}_{L}$. Up to some normalization multiplicative coefficients, the eigenfunctions of Problem (\ref{EigenPb}) are given by 
\[
\zeta(y,z)=\cos((2n+1)\pi z/(2L))\varphi_1(y),\ n\ge 0\qquad\zeta(y,z)=\cos(n\pi z/L)\varphi_j(y),\ n\ge0,\,j\ge 2.
\]
We deduce that $\mu_1=\min(\lambda_1+\pi^2/(2L)^2,\lambda_2)$. 
\end{proof}
\noindent Applied to $\gamma\in\mX$, estimate (\ref{PoincareMixt}) gives 
\begin{equation}\label{relation1bis}
\int_{\mathcal{R}_{L}}|\nabla \gamma|^2 - k^2|\gamma|^2\,dx \ge (\mu_1- k^2)\int_{\mathcal{R}_{L}} |\gamma|^2\,dx.
\end{equation}
Since $\gamma=u_s$ in $\Om_{L}$ and $\gamma=-u_i$ in $\mathcal{O}$, this implies
\begin{equation}\label{estimSoft2}
\int_{\Om_L}|\nabla u_s|^2 - k^2|u_s|^2\,dx+\int_{\mathcal{O}}|\nabla u_i|^2 - k^2|u_i|^2\,dx  \ge  (\mu_1- k^2)\left(\int_{\Om_L} |u_s|^2\,dx+\int_{\mathcal{O}} |u_i|^2\,dx\right).
\end{equation}
$\star$ Using (\ref{estim1}), (\ref{estimSoft2}) in (\ref{NullEnergySoft}), we obtain 
\begin{equation}\label{antepenul}
0  \ge  (\mu_1-k^2)\Big( \dsp\int_{\Om}|u_s|^2\,dx+\int_{\mathcal{O}} |u_i|^2\,dx\Big).
\end{equation}
Therefore, for $k^2<\mu_1$, since the interior of $\mathcal{O}$ is non empty, we are led to a contradiction and we must have $T\ne1$.
\end{proof}
\begin{remark}In Section \ref{sectionPerturbation}, we have seen how to construct sound soft obstacles (the flies) such that the reflection coefficient $R$ in (\ref{DecompoChampTotalObstruction}) verifies $R=0$. From the conservation of energy $|R|^2+|T|^2=1$, this implies $|T|=1$. Proposition \ref{MainPropositionMixt} shows that for $k^2<\min(\lambda_1+\pi^2/(2L)^2,\lambda_2)$, where $L>0$ denotes the smallest number such that the flies are located in the region $\om\times(-L;L)$, we cannot have $T=1$. This means that for such $k^2$, the wave suffers a phase shift after passing the defects.\end{remark}
\begin{remark}For the problem considered in this article (see (\ref{PbChampTotalTER})), $u$ is called a trapped mode if it satisfies $-\Delta u=k^2u$ in $\Om$, $u=0$ on $\partial\Om$ and $u\in\mH^1(\Om)$ (which implies that $u$ is exponentially decaying both at $\pm\infty$). Following an approach similar to the proof of Proposition \ref{MainPropositionMixt}, we can show that that there are no trapped modes for $k^2\in(\lambda_1;\min(\lambda_1+\pi^2/(2L)^2,\lambda_2))$. As a consequence, for a given $k^2\in(\lambda_1;\lambda_2)$, an obstacle cannot be completely invisible ($T=1$) and cannot trapped waves if it is too short in the $(Ox)$ direction ($L$ small).
\end{remark}

\section{Perturbation of the size of the flies}\label{SectionSizeOfFlies}
In Section \ref{sectionPerturbation}, we played with the position of the flies to cancel the reflection coefficient. In this section, we change a bit the point of view considering that the flies can modify slightly their size. In \S\ref{paraOneWave}, we explain how to cancel the  reflection coefficient when there is only one incident wave (as in the previous sections). Then, in \S\ref{paraSeveralWave}, we show how to impose reflection invisibility at higher frequency when there are several incident waves. To proceed, we will need more that two flies.

\subsection{One incident wave}\label{paraOneWave}

For $n=1,2$, assume that the shape of the fly located at $M_n=(y_n,z_n)$ changes from $\mathcal{O}$ to $(1+\tau_n\eps)\mathcal{O}:=\{\xi\in\R^3\,|\,(1+\tau_n\eps)^{-1}\xi\in\mathcal{O}\}$ where $\eps>0$ is small and where $\tau_n\in\R$ is a parameter to tune. Then, in the model considered above, the fly located at $M_n$ coincides with the set 
\[
\mathcal{O}^\eps_n(\tau_{n})  :=  \{x\in\R^3\,|\,\eps^{-1}(x-M_n)\in(1+\tau_n\eps)\mathcal{O}\}.
\]
In order to achieve reflection invisibility, in the following, we shall also need to adjust the position (but not the size) of two additional flies. For $n=3,4$, set  $\mathcal{O}^\eps_n:=\{x\in\R^3\,|\,\eps^{-1}(x-M_n)\in\mathcal{O}\}$ where $M_n=(y_n,z_n)$. Finally, for $\tau=(\tau_1,\tau_2)\in\R^2$, define the perturbed waveguide 
\[
\Om^{\eps}(\tau) :=  \Om^{0}\setminus \Big(\mathop{\cup}_{n=1}^2\overline{\mathcal{O}^\eps_n(\tau_{n})}\ \bigcup\  \mathop{\cup}_{n=3}^4\overline{\mathcal{O}^\eps_n}\Big).
\]
We denote $s^{\eps-}(\tau)$ the reflection coefficient for the scattering problem (\ref{PbChampTotalBIS}) considered in the geometry $\Om^{\eps}(\tau)$. Working as in Sections \ref{sectionAsympto}--\ref{sectionAsymptoCoeff}, we obtain the asymptotic expansion 
\begin{equation}\label{expansionSize1}
\begin{array}{l}
(4i\pi)^{-1}s^{\eps-}(\tau) =   \eps\Big(\ \dsp\sum_{n=1}^2w^+(M_n)^2\Capa((1+\tau_n\eps)\mathcal{O})+ \dsp\sum_{n=3}^4w^+(M_n)^2\Capa(\mathcal{O})\Big)\\[12pt]
\phantom{(4i\pi)^{-1}}+\eps^2\Big(\ \dsp\sum_{n=1}^2 u^{\eps}_{1}(M_n)w^+(M_n)\Capa((1+\tau_n\eps)\mathcal{O})+\dsp\sum_{n=3}^4 u_{1}(M_n)w^+(M_n)\Capa(\mathcal{O})\Big) +O(\eps^3),
\end{array} 
\end{equation}
where $u^{\eps}_{1}$, $u_{1}$ are respectively the solutions to the problems
\[
\begin{array}{|rcll}
\multicolumn{4}{|l}{\mbox{Find }u^{\eps}_{1}\in\mrm{H}^1_{\loc}(\Om^{0}) \mbox{ such that }u^{\eps}_{1}\mbox{ is outgoing and } }\\[3pt]
-\Delta u^{\eps}_{1}-k^2 u^{\eps}_{1}  & = & -\dsp\sum_{n=1}^{4} \Big([\Delta,\zeta_n]+k^2\zeta_n\mrm{Id}\Big)\Big(u_{0}(M_n)\,\frac{\Capa((1+\tau_n\eps)\mathcal{O})}{|x-M_n|}\Big) & \mbox{ in }\Om^{0}\\[12pt]
 u^{\eps}_{1}  & = & 0  & \mbox{ on }\Gamma^{0},
\end{array}
\]
\[
\begin{array}{|rcll}
\multicolumn{4}{|l}{\mbox{Find }u_{1}\in\mrm{H}^1_{\loc}(\Om^{0}) \mbox{ such that }u_{1}\mbox{ is outgoing and } }\\[3pt]
-\Delta u_{1}-k^2 u_{1}  & = & -\dsp\sum_{n=1}^{4} \Big([\Delta,\zeta_n]+k^2\zeta_n\mrm{Id}\Big)\Big(u_{0}(M_n)\,\frac{\Capa(\mathcal{O})}{|x-M_n|}\Big) & \mbox{ in }\Om^{0}\\[12pt]
 u_{1}  & = & 0  & \mbox{ on }\Gamma^{0}.
\end{array}
\]
For $n=1,2$, denote $W^\eps_{\tau_n}$ the harmonic potential of $(1+\tau_n\eps)\mathcal{O}$ ($W^\eps_{\tau_n}$ is the harmonic function of $\R^3\setminus\overline{(1+\tau_n\eps)}$ which vanishes at infinity and verifies $W^\eps_{\tau_n}=1$ on $(1+\tau_n\eps)\partial\mathcal{O}$).  Using the definition of the harmonic capacity (see e.g. \cite{PoSz51}) 
\begin{equation}\label{expansionCapacity}
\Capa((1+\tau_n\eps)\mathcal{O}):=\cfrac{1}{4\pi}\int_{\R^3\setminus\overline{(1+\tau_n\eps)\mathcal{O}}} |\nabla W^\eps_{\tau_n}(\xi)|^2\,d\xi,
\end{equation}
we obtain $\Capa((1+\tau_n\eps)\mathcal{O})=(1+\tau_n\eps)\Capa(\mathcal{O})$. As a consequence, we deduce that $u^{\eps}_{1}=(1+\tau_n\eps)u_{1}$. Plugging the two latter equalities in (\ref{expansionSize1}) yields 
\begin{equation}\label{expansionSize2}
  \cfrac{s^{\eps-}(\tau)}{4i\pi\,\Capa(\mathcal{O})} =\eps\dsp\sum_{n=1}^4w^+(M_n)^2 +\eps^2\Big(\dsp\sum_{n=1}^2 \tau_n w^+(M_n)^2+\dsp\sum_{n=1}^4 u_{1}(M_n)w^+(M_n)\Big) +\eps^3\,\tilde{s}^{\eps-}(\tau),
\end{equation}
where $\tilde{s}^{\eps-}(\tau)$ is a remainder. Now, we explain how to choose the positions and the sizes of the flies to obtain $s^{\eps-}(\tau)=0$. We remind the reader that, according to (\ref{defModes}), we have $w^{+}(M_n)= (2\beta_1)^{-1/2} e^{i \beta_1 z_n}\varphi_1(y_n)$. First, take $y_1$, $y_2$, $y_3$, $y_4\in\om$ such that $\varphi_1(y_1)=\varphi_1(y_2)=\varphi_1(y_3)=\varphi_1(y_4)$. Then, set $z_1$, $z_2$ such that $e^{2i \beta_1 z_1}=1$, $e^{2i \beta_1 z_2}=i$ and $z_3=(2m+1)\pi/(2\beta_1)+z_1$, $z_4=(2m+1)\pi/(2\beta_1)+z_2$. Here, $m\in\N$ is fixed so that the points $M_1,\dots, M_4$ are all distinct. Such a choice allows one to cancel the term of order $\eps$ in (\ref{expansionSize2}). Let us look for $\tau=(\tau_1,\tau_2)^{\top}$ under the form
\begin{equation}\label{DecompoTauSize}
\begin{array}{ll}
\tau = & \dsp\phantom{+ }2\beta_1\varphi_1(y_1)^{-2}\,\Big(\kappa_1-\Re e\,\sum_{n=1}^4 u_{1}(M_n)w^+(M_n)\Big) (1,0 )^{\top}\\[6pt]
  & \dsp+ 
2\beta_1\varphi_1(y_1)^{-2}\,
\Big(\kappa_2-\Im m\,\sum_{n=1}^4 u_{1}(M_n)w^+(M_n)\Big)(0,1)^{\top},
\end{array}
\end{equation}
where $\kappa_1$, $\kappa_2$ are some real parameters to determine. To impose $s^{\eps-}(\tau)=0$, plugging (\ref{DecompoTauSize}) in (\ref{expansionSize2}), we find that $(\kappa_1,\kappa_2)$ must be a solution to the problem 
\begin{equation}\label{PbFixedPointSize}
\begin{array}{|l}
\mbox{Find }\kappa=(\kappa_1,\kappa_2)\in\R^{2}\mbox{ such that }\kappa = \mathscr{F}^{\eps}(\kappa),
\end{array}~\\[5pt]
\end{equation}
with 
\begin{equation}\label{DefMapContraSize}
\mathscr{F}^{\eps}(\kappa):= -\eps\,(\Re e\,\tilde{s}^{\eps-}(\tau),\,\Im m\,\tilde{s}^{\eps-}(\tau))^{\top}. 
\end{equation}
Working as in Section \ref{JustiAsymptotics}, we can prove that for any given parameter $\gamma>0$, there is some $\eps_0>0$ such that for all $\eps\in(0;\eps_0]$, the map $\mathscr{F}^{\eps}$ is a contraction of $\mrm{B}^{\gamma}_2(O)=\{\kappa\in\R^{2}\,\big|\,|\kappa| \le \gamma\}$. Therefore, the Banach fixed-point theorem guarantees the existence of some $\eps_0>0$ such that for all $\eps\in(0;\eps_0]$, Problem (\ref{PbFixedPointSize}) has a unique solution in $\mrm{B}^{\gamma}_2(O)$. 
\begin{remark}
With two flies only, one can check that it is impossible in (\ref{expansionSize1}) to both cancel the term of order $\eps$ and to use the one of order $\eps^2$ to compensate the whole expansion. This is the reason why we need to add at least one passive fly. Note that here, we disposed two passive flies in the waveguide (the ones located at $M_3$, $M_4$) only to obtain a simple fixed point equation in (\ref{DefMapContraSize}). The scheme can also be implemented with one passive fly. For example, take $y_1$, $y_2$, $y_3 \in\om$ such that $\varphi_1(y_1)=\varphi_1(y_2)=\varphi_1(y_3)$ and $z_1$, $z_2$, $z_3\in\R$ such that $e^{2i \beta_1 z_1}=1$, $e^{2i \beta_1 z_2}=e^{2i\pi/3}$, $e^{2i \beta_1 z_3}=e^{4i\pi/3}$. With such a choice, first we cancel the term of order $\eps$ in (\ref{expansionSize2}). Then proceeding like in (\ref{DecompoTauSize}), (\ref{PbFixedPointSize}), we can derive a fixed point equation which admits a solution and from this solution we can find values for the parameters $\tau_1,\tau_2$ in (\ref{expansionSize2}) to achieve $s^{\eps-}(\tau)=0$.
\end{remark}

\subsection{Several incident waves}\label{paraSeveralWave}
Assume that the wavenumber $k$ verifies 
\[
\lambda_{J} < k^2 < \lambda_{J+1}.
\]
for some $J\in \N^{\ast}$. In this case, there are $2J$ propagating waves $w^{\pm}_{1},\dots,w^{\pm}_{J}$ in the waveguide. Let $\Om^{\eps}$ be a waveguide obtained from $\Om^{0}$ adding a finite number of flies. For $j=1,\dots,J$, denote $u^{\eps}_{j}$ the solution to the following scattering problem
\begin{equation}\label{PbChampTotalNwaves}
\begin{array}{|rcll}
\multicolumn{4}{|l}{\mbox{Find }u^{\eps}_j\in\mrm{H}^1_{\loc}(\Om^{\eps}) \mbox{ such that }u^{\eps}_j-w^{+}_{j}\mbox{ is outgoing and } }\\[3pt]
-\Delta u^{\eps}_j  & = & k^2 u^{\eps}_j & \mbox{ in }\Om^{\eps}\\[3pt]
 u^{\eps}_j  & = & 0  & \mbox{ on }\Gamma^{\eps}
\end{array}
\end{equation}
for $\eps$ small enough. In this paragraph, the sentence ``$u^{\eps}_j-w^{+}_{j}\in\mrm{H}^1_{\loc}(\Om^{\eps})$ is outgoing'' means that there holds the decomposition 
\begin{equation}\label{DecompoChampTotalMultiWaves}
u^{\eps}_j-w^{+}_{j} = \chi^+ \sum_{j'=1}^{J} s^{\eps +}_{j j'} \ w^+_{j'} + \chi^- \sum_{j'=1}^{J} s^{\eps -}_{j j'} \ w^-_{j'} + \tilde{u}^{\eps}_j,
\end{equation}
where $s^{\eps\pm}_{j j'}\in\Cplx$ and where $\tilde{u}^{\eps}_{j}\in\mH^1(\Om^{\eps})$ is a term exponentially decaying at $\pm\infty$. Our goal is to find a domain $\Om^{\eps}$ such that there holds $s^{\eps -}_{j j'}=0$ for $j,j'=1,\dots,J$. In such a situation, for any combination of the incident plane waves $w^{+}_{1},\dots,w^{+}_{J}$, the resulting scattered field is exponentially decaying at $-\infty$ so that the flies are invisible to an observer measuring the scattered field at $z=- R$ for large $R$.\\
\newline
It is known that the matrix made of the $s^{\eps -}_{j j'}$ is symmetric. Therefore, there are $P:=J(J+1)/2$ degrees of freedom and we have to cancel $2P$ real coefficients. To proceed, we will play with the size of $2P$ flies. For $n=1,\dots,2P$, assume that the fly located at $M_n=(y_n,z_n)$ coincides with the set 
\[
\mathcal{O}^\eps_n(\tau_{n})  :=  \{x\in\R^3\,|\,\eps^{-1}(x-M_n)\in(1+\tau_n\eps)\mathcal{O}\},
\]
where $\eps>0$ is small and where $\tau_n\in\R$ is a parameter to tune. In order to achieve reflection invisibility, as in the previous paragraph, we also need to adjust the position (but not the size) of additional flies. For $n=2P+1,\dots,N$, set $\mathcal{O}^\eps_n:=\{x\in\R^3\,|\,\eps^{-1}(x-M_n)\in\mathcal{O}\}$ where $M_n=(y_n,z_n)$. The choice of the parameter $N$ will be clarified later. Then, for $\tau:=(\tau_1,\dots,\tau_{2P})\in\R^{2P}$, define the perturbed waveguide 
\[
\Om^{\eps}(\tau) :=  \Om^{0}\setminus \Big(\mathop{\cup}_{n=1}^{2P}\overline{\mathcal{O}^\eps_n(\tau_{n})}\ \bigcup\ \dsp\mathop{\cup}_{n=2P+1}^N\overline{\mathcal{O}^\eps_n}\Big).
\]
Finally, we denote $s^{\eps -}_{j j'}(\tau)$ the reflection coefficient for the scattering problem (\ref{PbChampTotalNwaves}) considered in the geometry $\Om^{\eps}(\tau)$. Working as in Sections \ref{sectionAsympto}--\ref{sectionAsymptoCoeff} and \S\ref{paraOneWave}, we obtain the asymptotic expansion 
\begin{equation}\label{expansionMultiPerturbed}
\begin{array}{ll}
 & (4i\pi\,\Capa(\mathcal{O}))^{-1}\,s^{\eps -}_{j j'}(\tau)\\[4pt] 
 = & \eps\dsp\sum_{n=1}^N w^+_j(M_n)w^+_{j'}(M_n) +\eps^2\Big(\dsp\sum_{n=1}^{2P} \tau_n w^+_j(M_n)w^+_{j'}(M_n)+\dsp\sum_{n=1}^N u_{j,\,1}(M_n)w^+_{j'}(M_n)\Big) +\eps^3\,\tilde{s}^{\eps -}_{j j'}(\tau),
\end{array} 
\end{equation}
where $\tilde{s}^{\eps -}_{j j'}(\tau)$ is a remainder and where $u_{j,\,1}$ is the solution to the problem
\[
\begin{array}{|rcll}
\multicolumn{4}{|l}{\mbox{Find }u_{j,\,1}\in\mrm{H}^1_{\loc}(\Om^{0}) \mbox{ such that }u_{j,\,1}\mbox{ is outgoing and } }\\[3pt]
-\Delta u_{j,\,1}-k^2 u_{j,\,1}  & = & -\dsp\sum_{n=1}^{N} \left([\Delta,\zeta_n]+k^2\zeta_n\mrm{Id}\right)\left(w^+_j(M_n)\,\frac{\Capa(\mathcal{O})}{|x-M_n|}\right) & \mbox{ in }\Om^{0}\\[12pt]
 u_{j,\,1}  & = & 0  & \mbox{ on }\Gamma^{0}.
\end{array}
\] 
Now, we explain how to choose the positions and the sizes of the flies to obtain $s^{\eps -}_{j j'}(\tau)=0$. To proceed, as in the previous paragraph, we need to find positions for the flies such that the term of order $\eps$ in (\ref{expansionMultiPerturbed}) vanishes. But in the same time, we also wish to use the term of order $\eps^2$ to cancel the complete expansion. Let us translate this into equations. We remind the reader that there holds $w^{+}_j(M_n)= (2\beta_j)^{-1/2} e^{i \beta_j z_n}\varphi_j(y_n)$ with $\beta_j=(k^2-\lambda_j)^{1/2}$ (see (\ref{defModes})). First, take $y_1,\dots,y_N\in\om$ such that $y_1=\dots=y_N$. We want to impose 
\begin{equation}\label{identityFirstOrder}
\sum_{n=1}^N w^+_j(M_n)w^+_{j'}(M_n)=0\qquad\Leftrightarrow \qquad 
\dsp\sum_{n=1}^N e^{ i(\beta_j+\beta_{j'}) z_n}=0,\qquad\mbox{ for }1\le j\le j' \le J.
\end{equation}
To simplify, we assume that the wavenumber $k\in(\lambda_{J};\lambda_{J+1})$ is such that the numbers $\beta_j+\beta_{j'}=(k^2-\lambda_j)^{1/2}+(k^2-\lambda_{j'})^{1/2}$, $1\le j\le j' \le J$, are all distinct. Remark that, using the principle of isolated zeros, we can prove that wavenumbers such that this assumption is not verified form a set which is discrete or empty. Introduce $\gamma_1<\dots<\gamma_{P}$ such that $\{\gamma_p\}_{1\le p\le P}=\{\beta_j+\beta_{j'}\}_{1\le j\le j'\le J}$. In order to use the parameters $\tau_1,\dots,\tau_{2P}$ to cancel the whole expansion in (\ref{expansionMultiPerturbed}), we need to find $z_1,\dots,z_{2P}\in\R$ (distinct) such that the matrix  
\begin{equation}\label{defMatrix}
\mathscr{B}:=\left(\begin{array}{cccc}
\cos(\gamma_1 z_1) & \cos(\gamma_1 z_2) &\dots &\cos(\gamma_1 z_{2P})\\[3pt]
\sin(\gamma_1 z_1) & \sin(\gamma_1 z_2) &\dots &\sin(\gamma_1 z_{2P})\\[3pt]
\vdots & \vdots & \ddots & \vdots \\[3pt]
\cos(\gamma_P z_1) & \cos(\gamma_P z_2) &\dots &\cos(\gamma_P z_{2P})\\[3pt]
\sin(\gamma_P z_1) & \sin(\gamma_P z_2) &\dots &\sin(\gamma_P z_{2P})\\[3pt]
\end{array}
\right)
\end{equation}
is invertible. Assume for a moment that we have constructed $z_1,\dots,z_{N}\in\R$ such that (\ref{identityFirstOrder}) holds and such that the matrix $\mathscr{B}$ is invertible. Then, let us look for $\tau$ under the form $\tau=(\tau_1,\dots,\tau_{2P})^{\top}$ with, for $n=1,\dots,P$,
\begin{equation}\label{DecompoTauSizeMulti}
\begin{array}{ll}
\tau_{2n-1} =  2(\beta_j\beta_{j'})^{1/2}(\varphi_j(y_1)\varphi_{j'}(y_1))^{-1}\,\Big(\kappa_{2n-1}-\Re e\,\dsp\sum_{n=1}^N u_{j,\,1}(M_n)w^+_{j'}(M_n)\Big), \\[6pt]
\tau_{2n} = 2(\beta_j\beta_{j'})^{1/2}(\varphi_j(y_1)\varphi_{j'}(y_1))^{-1}\,\Big(\kappa_{2n}-\Im m\,\dsp\sum_{n=1}^N u_{j,\,1}(M_n)w^+_{j'}(M_n)\Big). 
\end{array}
\end{equation}
Here, $\kappa_1,\dots \kappa_{2P}$ are some real parameters to determine and $j\le j'$ are the indices such that $\gamma_n=\beta_j+\beta_{j'}$ (note that there is a one-to-one correspondence between the index $n$ and the pair $(j,j')$ for $j\le j'$). To impose $s^{\eps-}_{jj'}(\tau)=0$ for $1\le j\le j'\le J$, plugging (\ref{DecompoTauSizeMulti}) in (\ref{expansionMultiPerturbed}), we obtain that $\kappa:=(\kappa_1,\dots,\kappa_{2P})^{\top}$ must be a solution to the problem 
\begin{equation}\label{PbFixedPointSizeMulti}
\begin{array}{|l}
\mbox{Find }\kappa \in\R^{2P}\mbox{ such that }\kappa = \mathscr{F}^{\eps}(\kappa),
\end{array}~\\[5pt]
\end{equation}
with 
\begin{equation}\label{DefMapContraSizeMulti}
\mathscr{F}^{\eps}(\kappa):= -\eps\,\mathscr{B}^{-1}U. 
\end{equation}
In (\ref{DefMapContraSizeMulti}), $U\in\R^{2P}$ denotes the vector such that, for $n=1,\dots,P$, $U_{2n-1}=\Re e\,\tilde{s}^{\eps-}(\tau)_{jj'}$ and $U_{2n}=\Im m\,\tilde{s}^{\eps-}(\tau)_{jj'}$. Again, here $j\le j'$ are the indices such that $\gamma_n=\beta_j+\beta_{j'}$. Working as in  Section \ref{JustiAsymptotics}, we can prove that for any given parameter $\gamma>0$, there is some $\eps_0>0$ such that for all $\eps\in(0;\eps_0]$, the map $\mathscr{F}^{\eps}$ is a contraction of $\mrm{B}^{\gamma}_{2P}(O):=\{\kappa\in\R^{2P}\,\big|\,|\kappa| \le \gamma\}$. Therefore, the Banach fixed-point theorem guarantees the existence of some $\eps_0>0$ such that for all $\eps\in(0;\eps_0]$, Problem (\ref{PbFixedPointSizeMulti}) has a unique solution in $\mrm{B}^{\gamma}_{2P}(O)$.\\ 
\newline
Now we explain how to construct $z_1,\dots,z_{2P}\in\R$ such that the matrix $\mathscr{B}$ is invertible. We use a recursive approach. First, take $z_1$, $z_2$ such that 
\[
\mathscr{B}_{2}:=\left(\begin{array}{cc}
\cos(\gamma_1 z_1) & \cos(\gamma_1 z_2) \\[3pt]
\sin(\gamma_1 z_1) & \sin(\gamma_1 z_2) 
\end{array}
\right)
\]
is invertible. Then, let us show that we can find $z_3\in\R$ such that 
\[
\mathscr{B}_{3}:=\left(\begin{array}{ccc}
\cos(\gamma_1 z_1) & \cos(\gamma_1 z_2) & \cos(\gamma_1 z_3) \\[3pt]
\sin(\gamma_1 z_1) & \sin(\gamma_1 z_2) & \sin(\gamma_1 z_3) \\[3pt]
\cos(\gamma_2 z_1) & \cos(\gamma_2 z_2) & \cos(\gamma_2 z_3) 
\end{array}
\right)
\]
is invertible. The map $z_3\mapsto \mrm{det}(\mathscr{B}_{3})$ is analytic in $\Cplx$. Since $\gamma_2>\gamma_1$ and since $\mrm{det}(\mathscr{B}_{2})\ne0$, using Cramer's rule, we can prove that $\mrm{det}(\mathscr{B}_{3})\ne0$ for $z_3=iL $ with $L>0$ large enough. According to the principle of isolated zeros, we deduce that there is some $z_3\in \R$ (different from $z_1$, $z_2$) such that $\mathscr{B}_{3}$ is invertible. Then, define 
\[
\mathscr{B}_{4}:=\left(\begin{array}{cccc}
\cos(\gamma_1 z_1) & \cos(\gamma_1 z_2) & \cos(\gamma_1 z_3)& \cos(\gamma_1 z_4)  \\[3pt]
\sin(\gamma_1 z_1) & \sin(\gamma_1 z_2) & \sin(\gamma_1 z_3) & \sin(\gamma_1 z_4)  \\[3pt]
\cos(\gamma_2 z_1) & \cos(\gamma_2 z_2) & \cos(\gamma_2 z_3)  & \cos(\gamma_2 z_4)  \\[3pt]
\sin(\gamma_2 z_1) & \sin(\gamma_2 z_2) & \sin(\gamma_2 z_3)  & \sin(\gamma_2 z_4) 
\end{array}
\right).
\]
The map $z_4\mapsto \mrm{det}(\mathscr{B}_{4})$ is analytic in $\Cplx$. Take $z_4=iL $ with $L>0$ and pick the last column of $\mathscr{B}_{4}$ to compute $\mrm{det}(\mathscr{B}_{4})$ with Cramer's rule. Observe that $\sin(\gamma_2 iL)$ is purely imaginary whereas $\cos(\gamma_2 iL)$ is purely real. Using also that $\gamma_2>\gamma_1$ and that $\mrm{det}(\mathscr{B}_{3})$ is a non zero real number, we can prove that $\Im m\,(\mrm{det}(\mathscr{B}_{4}))\ne0$ for $L$ large enough. According to the principle of isolated zeros, we deduce that there is some $z_4\in \R$ (different from $z_1,\dots,z_3$) such that $\mathscr{B}_{4}$ is invertible. Continuing the process, we can find $z_1,\dots,z_{2P}$ such that the matrix $\mathscr{B}$ defined in (\ref{defMatrix}) is invertible.\\
\newline
Then, we want to determine $z_{2P+1},\dots,z_{N}$ ($N$ can be chosen as we wish) such that
\begin{equation}\label{RelationToSatisfy}
\dsp\sum_{n=1}^N e^{i\gamma_p z_n}=0,\qquad \mbox{for $p=1,\dots,P$}.
\end{equation}
For $n=1,\dots,2P$, set $z_{2P+n}=z_n+(2m_1+1)\pi/\gamma_1$ where $m_1$ is chosen so that $z_1,\dots,z_{4P}$ are all distinct. With this choice, we have
\[
\dsp\sum_{n=1}^{4P} e^{i\gamma_1 z_n}=0.
\]
For $n=1,\dots,4P$, set $z_{4P+n}=z_n+(2m_2+1)\pi/\gamma_2$, where $m_2$ is chosen so that $z_1,\dots,z_{8P}$ are all distinct. Then, we obtain 
\[
\dsp\sum_{n=1}^{8P} e^{i\gamma_p z_n}=0,\qquad \mbox{for $p=1,2$}.
\]
With this approach, we can find $N=2^{P+1}P$ numbers $z_1,\dots,z_N$ such that (\ref{RelationToSatisfy}) is satisfied. The technique is attractive because it is systematic and simple to implement. However, it requires a very high number of flies. For example, with $J=5$ (in this case $P=J(J+1)/2=15$), we find $N=983040$. It would be interesting to find alternative algorithms which are less flies consuming.
\begin{remark}
Above, we assumed that the wavenumber $k\in(\lambda_{J};\lambda_{J+1})$ is such that the numbers $\beta_j+\beta_{j'}=(k^2-\lambda_j)^{1/2}+(k^2-\lambda_{j'})^{1/2}$, $1\le j\le j' \le J$, are all distinct. It is an open problem to impose reflection invisibility when this assumption is not satisfied.\end{remark}

\section{Conclusion}\label{SectionConclu}
We explained how flies (small Dirichlet obstacles) should arrange to become invisible to an observer sending waves from $-\infty$ and measuring the resulting scattered field at the same position (see in particular Proposition \ref{propositionMainResult}). In other words, we constructed waveguides where the reflection coefficient $R$ satisfies $R=0$. We investigated a 3D setting. For 2D problems, the asymptotic calculus is a bit different, with the presence of a logarithm, but the analysis should be essentially the same. A possible direction to continue this work is to work with sound hard (Neumann) obstacles. We also considered the question of imposing $T=1$ ($T$ is the transmission coefficient). We observed that with small Dirichlet obstacles, it is impossible to have $T=1$. Moreover, we showed that for any sound soft obstacle (non necessarily small) embedded in the waveguide, we cannot have $T=1$ for wavenumbers smaller than an explicit value (see Proposition \ref{MainPropositionMixt}). It is an open question to know whether or not this bound is optimal.

\section{Appendix: justification of asymptotics}\label{JustiAsymptotics}

\noindent In this appendix, we explain how to justify the asymptotic expansion derived formally in Section \ref{sectionAsymptoCoeff}. More precisely, we wish to show estimate (\ref{MainEstimate}) which was the key ingredient to obtain Proposition \ref{propositionMainResult}. The statement of this estimate is as follows: for all $\vartheta>0$, there is some $\eps_0>0$ such that for all $\eps\in(0;\eps_0]$, there holds 
\[
|\tilde{s}^{\eps-}(\tau)-\tilde{s}^{\eps-}(\tau')| \le C\,|\tau-\tau'|,\qquad \forall\tau,\,\tau'\in \mrm{B}^{\vartheta}_3(O)=\{\tau\in\R^{3}\,\big|\,|\tau| \le \vartheta\},
\]
where $C>0$ is a constant independent of $\eps$. The proof will be divided into several steps and will be the concern of the next three paragraphs. We shall use the same notation as in Sections \ref{sectionAsymptoCoeff}, \ref{sectionPerturbation}.

\subsection{Explicit expression of the coefficient $\tilde{s}^{\eps-}(\tau)$}\label{paragraphDefiRemainder}
First, we provide an explicit formula for the coefficient $\tilde{s}^{\eps-}(\tau)$. According to (\ref{DvpCoefRef}), $\tilde{s}^{\eps-}(\tau)$ is defined as the remainder appearing in the decomposition
\begin{equation}\label{DvpCoefRefBis}
\frac{s^{\eps-}(\tau)}{4i\pi \,\Capa(\mathcal{O})}=\eps\dsp\sum_{n=1}^2w^+(M_n)^2+\eps^2\,\Big(2 w^+(M_1)\,\tau\cdot\nabla w^+(M_1)+\sum_{n=1}^2 u_{1}(M_n)w^+(M_n)\Big)+\eps^3\,\tilde{s}^{\eps-}(\tau).
\end{equation}
From (\ref{expressionReflecCoeff}), we know that 
\begin{equation}\label{expressionReflecCoeffPerturb}
i\,s^{\eps-}(\tau) = \dsp\int_{\Sigma^{L}} \frac{\partial (u^{\eps}(\tau)-w^+)}{\partial\nu}\,\overline{w^{-}}-(u^{\eps}(\tau)-w^+)\frac{\partial\overline{w^{-}}}{\partial\nu}\,d\sigma.
\end{equation}
Let us compute an asymptotic expansion of $u^{\eps}(\tau)$, the solution to Problem (\ref{PbChampTotalPerturb}), as $\eps$ tends to zero. We will work as in Section \ref{sectionAsympto} where we obtained an asymptotic expansion of $u^{\eps}(0)=u^{\eps}$. Consider the decomposition
\begin{equation}\label{DefAnsatzJustiPerturb}
\begin{array}{lcl}
u^{\eps}(\tau) & = & \phantom{+ \eps^m } w^{+} + \zeta_1(x)\,v_{0,\,1}(\eps^{-1}(x-M^{\tau}_1)) + \zeta_2(x)\,v_{0,\,2}(\eps^{-1}(x-M_2)) \\[10pt]
 & & + \eps\phantom{l}\Big(\dsp u_{1} + \zeta_1(x)\,\hat{v}_{1,\,1}(\eps^{-1}(x-M^{\tau}_1))+\zeta_2(x)\,v_{1,\,2}(\eps^{-1}(x-M_2))\Big) \\[10pt]
 & & + \eps^{2}\Big(\dsp u_{2} +\hat{u}_{2} + \zeta_1(x)\,\hat{v}_{2,\,1}(\eps^{-1}(x-M_1)) + \zeta_2(x)\,v_{2,\,2}(\eps^{-1}(x-M_2))\Big)+\eps^3\,\tilde{u}^{\eps}(\tau).
\end{array}
\end{equation}
In the above expression, the functions $v_{0,\,n}$, $u_{1}$, $v_{1,\,2}$, $u_{2}$, are respectively defined in (\ref{defSdnTerm0}), (\ref{PbChampTotalTerm1}), (\ref{defSdnTerm1}), (\ref{PbChampTotalu2}). In particular, $v_{0,\,1}(\eps^{-1}(x-M_1^{\tau})) = -w^+(M_1)\,W(\eps^{-1}(x-M_1^{\tau}))$. As $|\eps^{-1}(x-M_1)|$ becomes large, we have 
\begin{equation}\label{FormulaPerturb1}
\begin{array}{l}
W(\eps^{-1}(x-M_1^{\tau})) =  \dsp\frac{\Capa(\mathcal{O})}{|\eps^{-1}(x-M_1)-\tau|} + O(|\eps^{-1}(x-M_1)-\tau|^{-3})\\[20pt]
\phantom{W(\eps^{-1}(x-M_1^{\tau}))} =  \dsp\frac{\Capa(\mathcal{O})}{|\eps^{-1}(x-M_1)|} + \frac{\Capa(\mathcal{O})\,\tau\cdot\eps^{-1}(x-M_1)}{|\eps^{-1}(x-M_1)|^3}+ O(|\eps^{-1}(x-M_1)|^{-3}),
\end{array}
\end{equation}
where $W$ denotes the capacity potential already introduced in (\ref{defSdnTerm0}).  As a consequence, the term $u_1$ defined by (\ref{PbChampTotalTerm1}) indeed cancels the discrepancy 
\[
\dsp\sum_{n=1}^{2} \left([\Delta,\zeta_n]+k^2\zeta_n\mrm{Id}\right)\left(u_{0}(M_n)\,\frac{\Capa(\mathcal{O})}{|x-M_n|}\right)
\]
at order $\eps$. With $\hat{v}_{1,\,1}$, we impose the homogeneous Dirichlet boundary condition on $\partial \mathcal{O}^\eps_1(\tau)$ at order $\eps$. For $x\in\partial\mathcal{O}^{\eps}_{1}(\tau)$, we have 
\[
\begin{array}{lcl}
(w^{+} + v_{0,\,1}(\eps^{-1}(x-M_1^{\tau})) +\eps u_1)(x) &= &\eps u_1(M_1)+(x-M_1)\cdot\nabla w^{+}(M_1))+\dots \\[5pt]
&= &\eps(u_1(M_1)+(\tau+\eps^{-1}(x-M^{\tau}_1))\cdot\nabla w^{+}(M_1))+\dots\, . 
\end{array}
\]
Therefore, we take
\begin{equation}\label{defSdnTerm1Perturb}
\begin{array}{ll}
  & \hat{v}_{1,\,1}(\eps^{-1}(x-M_1^{\tau}))\\[5pt] 
= & -\Big((u_1(M_1)+\tau\cdot\nabla w^{+}(M_1))
\,W(\eps^{-1}(x-M_1^{\tau}))+\nabla w^+(M_1)\cdot\overrightarrow{W}(\eps^{-1}(x-M_1^{\tau}))\Big),
\end{array}
\end{equation}
where $\overrightarrow{W}$ is introduced in (\ref{defSdnTerm1}). After inserting $w^{+}(x) + v_{0,\,1}(\eps^{-1}(x-M_1^{\tau}))+v_{0,\,2}(\eps^{-1}(x-M_2))  +\eps(\dsp u_{1} + \hat{v}_{1,\,1}(\eps^{-1}(x-M_1^{\tau}))+v_{1,\,2}(\eps^{-1}(x-M_2)))$ into (\ref{PbChampTotalPerturb}), using formulas (\ref{defSdnTerm0}), (\ref{FarFieldOfNearField}), (\ref{AsymptoticBehaviour1}), (\ref{FormulaPerturb1})  and (\ref{defSdnTerm1Perturb}), we get the discrepancy
\begin{equation}\label{discrepancyModif}
\begin{array}{l}
\phantom{+}\eps^2\dsp\sum_{n=1}^{2} \left([\Delta,\zeta_n]+k^2\zeta_n\mrm{Id}\right)\left(u_{1}(M_n)\,\frac{\Capa(\mathcal{O})}{|x-M_n|}\right) \\[14pt]
 +\eps^2\dsp\left([\Delta,\zeta_1]+k^2\zeta_1\mrm{Id}\right)\left(w^+(M_1)\,
\frac{\Capa(\mathcal{O})\,\tau\cdot(x-M_1)}{|x-M_1|^3}+\tau\cdot\nabla w^{+}(M_1)\,\frac{\Capa(\mathcal{O})}{|x-M_1|}\right)+O(\eps^3).
\end{array}
\end{equation}
The term $u_2$, which satisfies Problem (\ref{PbChampTotalu2}), allows one to cancel the first part (first line) of this discrepancy. In (\ref{DefAnsatzJustiPerturb}), to deal with the second component of (\ref{discrepancyModif}) (second line), we introduced the function $\hat{u}_2$. We take $\hat{u}_2:=\hat{u}_{2a}+\hat{u}_{2b}$  where $\hat{u}_{2a}$, $\hat{u}_{2b}$ are the solutions to the problems
\begin{equation}\label{PbChampTotalu2Perturb}
\begin{array}{|rcll}
\multicolumn{4}{|l}{\mbox{Find }\hat{u}_{2a}\in\mH^1_{\loc}(\Om^{0}) \mbox{ such that } }\\[3pt]
-\Delta \hat{u}_{2a}-k^2 \hat{u}_{2a}  & = & f_a(\tau) & \mbox{ in }\Om^{0}\\[3pt]
 \hat{u}_{2a}  & = & 0  & \mbox{ on }\Gamma^{0}\\[3pt]
\multicolumn{4}{|l}{\phantom{-\Delta \hat{u}_{2a}-k^2}\hat{u}_{2a}\mbox{ is outgoing}}
\end{array}\qquad\qquad \begin{array}{|rcll}
\multicolumn{4}{|l}{\mbox{Find }\hat{u}_{2b}\in\mH^1_{\loc}(\Om^{0})\mbox{ such that }}\\[3pt]
-\Delta \hat{u}_{2b}-k^2 \hat{u}_{2b}  & = & f_b(\tau) & \mbox{ in }\Om^{0}\\[3pt]
 \hat{u}_{2b}  & = & 0  & \mbox{ on }\Gamma^{0}\\[3pt]
\multicolumn{4}{|l}{\phantom{-\Delta \hat{u}_{2a}-k^2}\hat{u}_{2b}\mbox{ is outgoing.}}
\end{array}
\end{equation}
Here, in accordance with (\ref{discrepancyModif}), the source terms $f_a(\tau)$, $f_b(\tau)$ are defined by
\[
\begin{array}{lcl}
f_a(\tau) & = & -\left([\Delta,\zeta_1]+k^2\zeta_1\mrm{Id}\right)\left(w^+(M_1)\,
\dsp\frac{\Capa(\mathcal{O})\,\tau\cdot(x-M_1)}{|x-M_1|^3}\right)\\[15pt]
f_b(\tau) & = & -\left([\Delta,\zeta_1]+k^2\zeta_1\mrm{Id}\right)\left(\tau\cdot\nabla w^{+}(M_1)\,\dsp\frac{\Capa(\mathcal{O})}{|x-M_1|}\right).
\end{array}
\]
Since Problems (\ref{PbChampTotalu2Perturb}) are linear, we obtain $\hat{u}_{2b}=\tau\cdot\nabla w^{+}(M_1)u_1/w^{+}(M_1)$ where $u_1$ refers to the function introduced in (\ref{PbChampTotalTerm1}). Plugging (\ref{DefAnsatzJustiPerturb}) in (\ref{expressionReflecCoeffPerturb}), we obtain
\begin{equation}\label{relationReflexionCoef}
\dsp  i\,s^{\eps-}(\tau)  = \dsp\int_{\Sigma^{\ell}} \frac{\partial }{\partial\nu} (\eps u_{1}+\eps^2(u_2+\hat{u}_2)+\eps^3\tilde{u}^{\eps}(\tau))\,\overline{w^{-}}- (\eps u_{1}+\eps^2(u_2+\hat{u}_2)+\eps^3\tilde{u}^{\eps}(\tau))\frac{\partial\overline{w^{-}}}{\partial\nu}\,d\sigma.
\end{equation}
Let us focus our attention on the term involving $\hat{u}_2=\hat{u}_{2a}+\hat{u}_{2b}$ in (\ref{relationReflexionCoef}). Since $\hat{u}_{2b}=\tau\cdot\nabla w^{+}(M_1)u_1/w^{+}(M_1)$, according to (\ref{IntegrationByPart1})--(\ref{ExpressionCoef1}), we have 
\begin{equation}\label{relationNum1}
\dsp\int_{\Sigma^{\ell}} \frac{\partial \hat{u}_{2b}}{\partial\nu}\,\overline{w^{-}}- \hat{u}_{2b}\frac{\partial\overline{w^{-}}}{\partial\nu}\,d\sigma=4i\pi\,\Capa(\mathcal{O}) (w^+(M_1)\,\tau\cdot\nabla w^+(M_1)).
\end{equation}
Now, we compute 
\[
I:=\dsp\int_{\Sigma^{\ell}} \frac{\partial \hat{u}_{2a}}{\partial\nu}\,\overline{w^{-}}-\hat{u}_{2a}\frac{\partial\overline{w^{-}}}{\partial\nu}\,d\sigma.
\]
Integrating by parts in $\om\times(-\ell;\,\ell)$ and using the equation $\Delta w^{-}+k^2w^{-}=0$ as well as (\ref{PbChampTotalu2Perturb}), we find, working exactly as in (\ref{IntegrationByPart1}), 
\begin{equation}\label{IntegrationByPart1Bis}
\begin{array}{lcl}
I & = & \dsp\,\int_{\om\times(-\ell;\,\ell) }\Delta \hat{u}_{2a} \,\overline{w^{-}}-\hat{u}_{2a} \,\Delta\overline{w^{-}}\,dx\\[12pt]
& = & \dsp\int_{\Om^0}\left(\Delta \hat{u}_{2a}+k^2 \hat{u}_{2a}\right) \,\overline{w^{-}}\,dx\\[12pt]
& = &  \dsp\int_{\Om^0}\left([\Delta,\zeta_1]+k^2\zeta_1\mrm{Id}\right)\left(w^+(M_1)\,
\dsp\frac{\Capa(\mathcal{O})\,\tau\cdot(x-M_1)}{|x-M_1|^3}\right)\,\overline{w^{-}}\,dx\\[12pt]
& = & w^+(M_1)\,\Capa(\mathcal{O})\dsp\int_{\Om^0}\overline{w^{-}}\,[\Delta,\zeta_1]\left(\frac{\tau\cdot(x-M_1)}{|x-M_1|^3}\right)\,-\frac{\zeta_1\tau\cdot(x-M_1)\Delta \overline{w^{-}}}{|x-M_1|^3}\,dx.
\end{array}
\end{equation}
Noticing that $[\Delta,\zeta_1](\tau\cdot(x-M_1)/|x-M_1|^3)$ vanishes in a neighbourhood of $M_1$ (see the discussion after (\ref{PbChampTotalTerm1})) and using the Lebesgue's dominated convergence theorem we can write
\begin{equation}\label{IntegrationByPartAutreBis}
\begin{array}{lcl}
I & = & \dsp\lim_{\delta\to 0}\ w^+(M_1)\,\Capa(\mathcal{O})\dsp\int_{\Om^{0\delta}}\overline{w^{-}}\,[\Delta,\zeta_1]\left(\frac{\tau\cdot(x-M_1)}{|x-M_1|^3}\right)\,-\frac{\zeta_1\tau\cdot(x-M_1)\Delta \overline{w^{-}}}{|x-M_1|^3}\,dx.
\end{array}
\end{equation}
In (\ref{IntegrationByPartAutreBis}), the set $\Om^{0\delta}$ is defined by  $\Om^{0\delta}:=\Om^0\setminus\overline{\mrm{B}^{\delta}_3(M_1)}$. Remark that $[\Delta,\zeta_1](\tau\cdot(x-M_1)/|x-M_1|^3)=\Delta(\zeta_1\tau\cdot(x-M_1)/|x-M_1|^3)$ in $\Om^{0\delta}$. Therefore, we have 
\begin{equation}\label{estimMultiples}
\begin{array}{l}
I =  \phantom{-}\dsp\lim_{\delta\to 0}\ w^+(M_1)\,\Capa(\mathcal{O})\dsp\int_{\Om^{0\delta}}\overline{w^{-}}\,\Delta\left(\frac{\zeta_1\tau\cdot(x-M_1)}{|x-M_1|^3}\right)\,-\frac{\zeta_1\tau\cdot(x-M_1)\Delta \overline{w^{-}}}{|x-M_1|^3}\,dx \\[15pt]
\phantom{I } =  -\dsp\lim_{\delta\to 0}\ w^+(M_1)\,\Capa(\mathcal{O})\dsp\int_{\Om^{0\delta}}\overline{w^{-}}\,\Delta\left(\zeta_1\tau\cdot\nabla\left(\frac{1}{|x-M_1|}\right)\right)\,-\zeta_1\tau\cdot\nabla\left(\frac{1}{|x-M_1|}\right)\Delta \overline{w^{-}}\,dx \\[15pt]
\phantom{I } =  -\dsp\lim_{\delta\to 0}\ w^+(M_1)\,\Capa(\mathcal{O})\dsp\int_{\Om^{0\delta}}\overline{w^{-}}\,\Delta\left(\tau\cdot\nabla\left(\frac{\zeta_1}{|x-M_1|}\right)\right)\,-\tau\cdot\nabla\left(\frac{\zeta_1}{|x-M_1|}\right)\Delta \overline{w^{-}}\,dx \\[15pt]
\phantom{I =}   +\dsp\lim_{\delta\to 0}\ w^+(M_1)\,\Capa(\mathcal{O})\dsp\int_{\Om^{0\delta}}\overline{w^{-}}\,\Delta\left(\frac{1}{|x-M_1|}\,\tau\cdot\nabla\,\zeta_1\right)\,-\frac{1}{|x-M_1|}\,\tau\cdot\nabla\,\zeta_1\,\Delta \overline{w^{-}}\,dx. 
\end{array}
\end{equation}
Integrate by parts in the two terms of the right hand side of (\ref{estimMultiples}). Using that $\zeta_1$ is compactly supported in $\Om^{0}$ and equal to one in a neighbourhood of $M_1$ (so that there holds $\tau\cdot\nabla\,\zeta_1=0$ in this region), we deduce
\begin{equation}\label{Integnabla}
\begin{array}{lcl}
I& = & -\dsp\lim_{\delta\to 0}\ w^+(M_1)\,\Capa(\mathcal{O})\dsp\int_{\Om^{0\delta}}\overline{w^{-}}\,\Delta\left(\tau\cdot\nabla\left(\frac{\zeta_1}{|x-M_1|}\right)\right)\,-\tau\cdot\nabla\left(\frac{\zeta_1}{|x-M_1|}\right)\Delta \overline{w^{-}}\,dx \\[15pt]
& = & \phantom{-}\dsp\lim_{\delta\to 0}\  w^+(M_1)\,\Capa(\mathcal{O})\dsp\int_{\Om^{0\delta}}\overline{\tau\cdot\nabla w^{-}}\,\Delta\left(\frac{\zeta_1}{|x-M_1|}\right)\,-\frac{\zeta_1}{|x-M_1|}\,\Delta(\tau\cdot\nabla \overline{w^{-}})\,dx \\[15pt]
& = & \phantom{-}\dsp\lim_{\delta\to 0} \ w^+(M_1)\,\Capa(\mathcal{O})\dsp\int_{\partial\mrm{B}^{\delta}_3(M_1)} \overline{\tau\cdot\nabla w^{-}}\,\partial_{\nu}(|x-M_n|^{-1})-|x-M_n|^{-1}\partial_{\nu}\overline{\tau\cdot\nabla w^{-}}\,d\sigma.
\end{array}
\end{equation}
In this expression, $\nu$ stands for the normal unit vector to $\partial\mrm{B}^{\delta}_3(M_1)$ directed to the interior of $\mrm{B}^{\delta}_3(M_1)$. Note that to obtain the second line of (\ref{Integnabla}), we use the relation
\[
\lim_{\delta\to 0}\ \dsp\int_{\partial\mrm{B}^{\delta}_3(M_1)}\frac{\zeta_1}{|x-M_1|}\,\Delta\overline{w^{-}}\,d\sigma=0.
\]
An explicit calculus similar to (\ref{explicitCalculus}) gives $I=4i\pi\,\Capa(\mathcal{O}) (w^+(M_1)\,\tau\cdot\nabla w^+(M_1))$. Gathering this result with (\ref{relationNum1}) leads to
\[
\dsp\int_{\Sigma^{\ell}} \frac{\partial \hat{u}_2}{\partial\nu}\,\overline{w^{-}}- \hat{u}_2\frac{\partial\overline{w^{-}}}{\partial\nu}\,d\sigma=4i\pi\,\Capa(\mathcal{O}) (2w^+(M_1)\,\tau\cdot\nabla w^+(M_1)).
\]
Plugging this identity into (\ref{relationReflexionCoef}), using (\ref{ExpressionCoef1}), (\ref{ExpressionCoef2}) and identifying with (\ref{DvpCoefRef}), we obtain
\begin{equation}\label{expressionReflecCoeffRemainder}
i\,\tilde{s}^{\eps-}(\tau) = \dsp\int_{\Sigma^{\ell}} \frac{\partial \tilde{u}^{\eps}(\tau)}{\partial\nu}\,\overline{w^{-}}- \tilde{u}^{\eps}(\tau)\frac{\partial\overline{w^{-}}}{\partial\nu}\,d\sigma.
\end{equation}
In this formula, $\tilde{u}^{\eps}(\tau)$ refers to the function appearing in the decomposition (\ref{DefAnsatzJustiPerturb}) $u^{\eps}(\tau)=\mathfrak{u}^{\eps}(\tau)+\eps^3\tilde{u}^{\eps}(\tau)$. Here, $\mathfrak{u}^{\eps}(\tau)$ stands for the sum of the terms of orders $\eps^0$, $\eps$, $\eps^2$ in (\ref{DefAnsatzJustiPerturb}). With this notation, $\tilde{u}^{\eps}(\tau)$ is the solution to the problem  
\begin{equation}\label{PbChampTotalu2Perturb}
\begin{array}{|rcll}
\multicolumn{4}{|l}{\mbox{Find }\tilde{u}^{\eps}(\tau)\in\mH^1_{\loc}(\Om^{\eps}(\tau)) \mbox{ such that }\tilde{u}^{\eps}(\tau)\mbox{ is outgoing and } }\\[3pt]
-\Delta \tilde{u}^{\eps}(\tau)-k^2 \tilde{u}^{\eps}(\tau)  & = & f^{\eps}(\tau) & \mbox{ in }\Om^{\eps}(\tau)\\[3pt]
 \tilde{u}^{\eps}(\tau)  & = & g^{\eps}(\tau)  & \mbox{ on }\partial\Om^{\eps}(\tau),
\end{array}
\end{equation}
where $f^{\eps}(\tau)= \eps^{-3}(-\Delta \mathfrak{u}^{\eps}(\tau)-k^2 \mathfrak{u}^{\eps}(\tau))$ and $g^{\eps}(\tau)=\eps^{-3}\mathfrak{u}^{\eps}(\tau)$. The next step to prove estimate (\ref{MainEstimate}) consists in showing that $\tilde{u}^{\eps}(\tau)$  is Lipschitz continuous with respect to the parameter $\tau$ with a constant independent of $\eps$. To proceed, we need to recall some standard material to study  problems like (\ref{PbChampTotalu2Perturb}).

\subsection{Solvability of  the problems in weighted spaces}
In the following, we will use a change of variables to compare the solutions to (\ref{PbChampTotalu2Perturb}) in the geometry $\Om^\eps(0)=\Om^\eps$.  Following the classical study \cite{Kond67} (see also \cite[Chap. 3 and 5]{NaPl94}), we introduce the weighted Sobolev space $\mW^1_{\beta}(\Om^\eps)$ (Kondrat'ev space) defined as the completion of $\mathscr{C}^{\infty}_0(\Om^0\setminus\cup_{n=1}^2\mathcal{O}^\eps_n)$ for the norm
\begin{equation}\label{NormKondra}
\|v\|_{\mW^1_{\beta}(\Om^\eps)} = (\|e^{\beta |z|}\nabla v\|^2_{\mrm{L}^2(\Om^\eps)}+\|e^{\beta |z|}v\|^2_{\mrm{L}^2(\Om^\eps)})^{1/2}.
\end{equation}
Here $\mathscr{C}^{\infty}_0(\Om^0\setminus(\cup_{n=1}^2\mathcal{O}^\eps_n))$ denotes the set of infinitely differentiable functions supported in $\Om^0\setminus\cup_{n=1}^2\mathcal{O}^\eps_n$ and $\beta\in\R$ is the weight exponent. The space $\mW^1_{\beta}(\Om^\eps)$ consists of functions of $\mH^1_{\loc}(\Om^{\eps})$ which vanish on $\Gamma^0$ (but not necessarily on $\partial\mathcal{O}^\eps_1\cup\partial\mathcal{O}^\eps_2$) with finite norm (\ref{NormKondra}) . Observe in particular that $\mW^1_{0}(\Om^\eps)=\{v\in\mH^1(\Om^{\eps})\,|\,v=0\mbox{ on }\Gamma^0\}$ and that for $\beta>0$, the functions of $\mW^1_{\beta}(\Om^\eps)$ decay exponentially at $z=\pm\infty$. We also define 
\[
\mathring{\mW}^1_{\beta}(\Om^\eps) := \{v\in \mW^1_{\beta}(\Om^\eps)\,|\,v=0\mbox{ on }\partial\mathcal{O}^\eps_1\cup\partial\mathcal{O}^\eps_2\}.
\]
In order to prescribe radiation conditions at $z=\pm\infty$ (as in (\ref{DecompoChampScattered})), for $\beta>0$ we introduce the space with detached asymptotic (see, e.g., the reviews \cite{Naza99a,Naza99b}) $\boldsymbol{\mW}^1_{-\beta}(\Om^{\eps})$ that consists of functions $v\in \mW^1_{-\beta}(\Om^{\eps})$ that admit the representation 
\[
v = \chi^+ s^{+} w^+ + \chi^- s^{-} w^- + \tilde{v}, 
\]
with coefficients $s^{\pm}\in\Cplx$ and remainder $\tilde{v}\in \mW^1_{\beta}(\Om^{\eps})$. This space is a Hilbert space for the inner product naturally associated with the norm 
\[
\|v\|_{\boldsymbol{\mW}^1_{-\beta}(\Om^{\eps})} = \Big( |s^{+}|^2+|s^{-}|^2+\|\tilde{v}\|^2_{\mW^1_{\beta}(\Om^{\eps})} \Big)^{1/2}.
\]
We define the map $\mrm{A}^{\eps}(0)$ such that
\begin{equation}\label{defOpGeomInitial}
\begin{array}{lccc}
\mrm{A}^{\eps}(0): & \boldsymbol{\mW}^1_{\beta}(\Om^{\eps}) & \longrightarrow & \mathring{\mW}^1_{-\beta}(\Om^{\eps})^{\ast}\times\mH^{1/2}(\partial\mathcal{O}^\eps_1\cup\partial\mathcal{O}^\eps_2)\\
 & u = \chi^+ s^{+} w^+ + \chi^- s^{-} w^- + \tilde{u} & \longmapsto & (f,g)\phantom{s\cup\partial\mathcal{O}^\eps_2} 
\end{array} 
\end{equation}
where $g=u|_{\partial\mathcal{O}^\eps_1\cup\partial\mathcal{O}^\eps_2}$ and where $f$ is the function such that
\[
\langle f,\overline{v}\rangle_{\Om^{\eps}} = -\int_{\Om^{\eps}} (\Delta+k^2\mrm{Id})(\chi^+ s^{+} w^+ + \chi^- s^{-} w^- )\overline{v}\,dx+\int_{\Om^{\eps}} \nabla\tilde{u}\cdot\nabla \overline{v}-k^2\tilde{u} \overline{v}\,dx,\quad\forall v\in\mathring{\mW}^1_{-\beta}(\Om^{\eps}).
\]
Here, $\mathring{\mW}^1_{-\beta}(\Om^{\eps})^{\ast}$ stands for the topological dual space to $\mathring{\mW}^1_{-\beta}(\Om^{\eps})$ while $\langle \cdot,\cdot\rangle_{\Om^{\eps}}$ corresponds to the duality pairing between $\mathring{\mW}^1_{-\beta}(\Om^{\eps})^{\ast}$ and $\mathring{\mW}^1_{-\beta}(\Om^{\eps})$. For all $\beta\in(0;\sqrt{\lambda_2-k^2})$, we can show that $\mrm{A}^{\eps}(0)$ is an isomorphism for $\eps$ small enough. To obtain this result, for example one can adapt the proof of \cite[Prop. 3.1]{ChCN15} (see also \cite[Chap. 4]{MaNP00}). Moreover, $\mrm{A}^{\eps}(0)^{-1}$ is uniformly bounded for $\eps\in(0;\eps_0]$.\\
\newline
As in (\ref{defOpGeomInitial}), we define the operator $\mrm{A}^{\eps}(\tau):  \boldsymbol{\mW}^1_{\beta}(\Om^{\eps}(\tau)) \rightarrow  \mathring{\mW}^1_{-\beta}(\Om^{\eps}(\tau))^{\ast}\times\mH^{1/2}(\partial\mathcal{O}^\eps_1(\tau)\cup\partial\mathcal{O}^\eps_2)$ where the spaces are the same as the ones introduced above with $\Om^{\eps}$ replaced by $\Om^{\eps}(\tau)$.

\subsection{Error estimate}

Now, we have all the tools to establish estimate (\ref{MainEstimate}). We will work as in the classical proofs of perturbations theory for linear operators (see \cite[Chap. 7, \S6.5]{Kato95}, \cite[Chap. 4]{HiPh57}). Let us consider some smooth diffeomorphism $\mathscr{L}(\eps\tau)$ which maps $\Om^{\eps}$ (defined in (\ref{defSetIni})) into $\Om^{\eps}(\tau)$. In a neighbourhood of $\mathcal{O}^\eps_1$, $\mathscr{L}(\eps\tau)$ coincides with the transformation $x\mapsto x+\eps\tau$. We can assume that the global change of variables is equal to the identity for $x$ such that $|z|\ge\ell/2$ and that its Jacobian matrix satisfies the relations
\begin{equation}\label{smallPerturbation}
\left(\cfrac{\partial (\mathscr{L}(\eps\tau)x)}{\partial x_k}\right)_{1\le k\le 3}=\mrm{Id}+\tilde{\mathscr{L}}^{\eps\tau}(x),\qquad \sum_{|\alpha|=p}\|\partial^{\alpha}_x \tilde{\mathscr{L}}^{\eps\tau}(x)\|_{\R^{3\times 3}} \le C_p \,\eps\tau,\qquad p=0,1,\dots\,.
\end{equation}
Here, $\mrm{Id}$ stands for the $3\times 3$ identity matrix and $\partial^{\alpha}_x$ is the standard multi-index notation used for derivatives. In other words, we assume that $\mathscr{L}$ is non singular and almost identical for small $\eps$. Define $\tilde{U}^{\eps}(\tau):=\tilde{u}^{\eps}(\tau)\circ \mathscr{L}(\eps\tau)$, $F^{\eps}(\tau):=f^{\eps}(\tau)\circ \mathscr{L}(\eps\tau)$, $G^{\eps}(\tau):=g^{\eps}(\tau)\circ \mathscr{L}(\eps\tau)$ where $\tilde{u}^{\eps}(\tau)$, $f^{\eps}(\tau)$, $g^{\eps}(\tau)$ are the functions appearing in (\ref{PbChampTotalu2Perturb}). In (\ref{DefAnsatzJustiPerturb}), we can choose $\hat{v}_{2,\,1}$, $v_{2,\,2}$ so that there holds
\begin{eqnarray}
\label{line1}\|(F^{\eps}(\tau),G^{\eps}(\tau))\|_{\mathring{\mW}^1_{-\beta}(\Om^{\eps})^{\ast}\times\mH^{1/2}(\partial\mathcal{O}^\eps_1\cup\partial\mathcal{O}^\eps_2)} & \le & C\\ 
\label{line2}\|(F^{\eps}(\tau),G^{\eps}(\tau))-(F^{\eps}(\tau'),G^{\eps}(\tau'))\|_{\mathring{\mW}^1_{-\beta}(\Om^{\eps})^{\ast}\times\mH^{1/2}(\partial\mathcal{O}^\eps_1\cup\partial\mathcal{O}^\eps_2)} & \le & C\,|\tau-\tau'|
\end{eqnarray}
for all $\eps\in(0;\eps_0]$, $\tau$, $\tau'\in\mrm{B}^{\vartheta}_3(O)=\{\tau\in\R^{3}\,\big|\,|\tau| \le \vartheta\}$. On the other hand, under the change of variables $\mathscr{L}(\eps\tau)$, $\mrm{A}^{\eps}(\tau):  \boldsymbol{\mW}^1_{\beta}(\Om^{\eps}(\tau)) \rightarrow  \mathring{\mW}^1_{-\beta}(\Om^{\eps}(\tau))^{\ast}\times\mH^{1/2}(\partial\mathcal{O}^\eps_1(\tau)\cup\partial\mathcal{O}^\eps_2)$ is transformed into the operator $\mathscr{A}^{\eps}(\tau):  \boldsymbol{\mW}^1_{\beta}(\Om^{\eps}) \rightarrow  \mathring{\mW}^1_{-\beta}(\Om^{\eps})^{\ast}\times\mH^{1/2}(\partial\mathcal{O}^\eps_1\cup\partial\mathcal{O}^\eps_2)$ that ``differs little'' (in a sense similar to (\ref{smallPerturbation})) from $\mrm{A}^{\eps}(0)$. Moreover, the coefficients of this differential operators depend smoothly on the parameter $\eps\tau\in(0;\eps_0]$. Therefore, we have the estimate 
\begin{equation}\label{smallOperator}
\|\mathscr{A}^{\eps}(\tau)-\mrm{A}^{\eps}(0)\|\le C\,\eps\tau
\end{equation}
where $\|\cdot\|$ refers to the usual norm for the linear operators acting from $\boldsymbol{\mW}^1_{\beta}(\Om^{\eps})$ to $\mathring{\mW}^1_{-\beta}(\Om^{\eps})^{\ast}\times\mH^{1/2}(\partial\mathcal{O}^\eps_1\cup\partial\mathcal{O}^\eps_2)$. Note that
\[
\begin{array}{lcl}
 \mrm{A}^{\eps}(\tau)\tilde{u}^{\eps}(\tau) = (f^{\eps}(\tau),g^{\eps}(\tau)) & \Leftrightarrow &   \mathscr{A}^{\eps}(\tau)\tilde{U}^{\eps}(\tau) = (F^{\eps}(\tau),G^{\eps}(\tau))\\[8pt]
 & \Leftrightarrow &  (\mrm{A}^{\eps}(0)+(\mathscr{A}^{\eps}(\tau)-\mrm{A}^{\eps}(0)))\tilde{U}^{\eps}(\tau) = (F^{\eps}(\tau),G^{\eps}(\tau)). 
\end{array}
\]
Since $\mrm{A}^{\eps}(0)^{-1}$ is uniformly bounded for $\eps\in(0;\eps_0]$, we deduce from (\ref{line1}), (\ref{smallOperator}) that $\|\tilde{U}^{\eps}(\tau)\|_{\boldsymbol{\mW}^1_{-\beta}(\Om^{\eps})} \le C$ for all $\tau\in\mrm{B}^{\vartheta}_3(O)$. Now, pick some $\tau,\tau'\in\mrm{B}^{\vartheta}_3(O)$. We have
\[
\begin{array}{ll}
 & \mathscr{A}^{\eps}(\tau)\tilde{U}^{\eps}(\tau)-\mathscr{A}^{\eps}(\tau')\tilde{U}^{\eps}(\tau') = (F^{\eps}(\tau)-F^{\eps}(\tau'),G^{\eps}(\tau)-G^{\eps}(\tau')) \\[8pt]
\Leftrightarrow & \mathscr{A}^{\eps}(\tau)(\tilde{U}^{\eps}(\tau)-\tilde{U}^{\eps}(\tau'))
=(\mathscr{A}^{\eps}(\tau')-\mathscr{A}^{\eps}(\tau))\tilde{U}^{\eps}(\tau') + (F^{\eps}(\tau)-F^{\eps}(\tau'),G^{\eps}(\tau)-G^{\eps}(\tau')) .
\end{array}
\]
Using (\ref{line2}), the estimate 
\[
\|\mathscr{A}^{\eps}(\tau)-\mathscr{A}^{\eps}(\tau')\|\le C\,\eps|\tau-\tau'| 
\]
 as well as the fact that $\mathscr{A}^{\eps}(\tau)$ is uniformly invertible for $\eps$ small enough, we obtain
\[
\|\tilde{U}^{\eps}(\tau)-\tilde{U}^{\eps}(\tau')\|_{\boldsymbol{\mW}^1_{-\beta}(\Om^{\eps})} \le C\,|\tau-\tau'|.
\] 
Finally, remarking that $\tilde{U}^{\eps}(\tau)=\tilde{u}^{\eps}(\tau)$ and $\tilde{U}^{\eps}(\tau')=\tilde{u}^{\eps}(\tau')$ for $x$ such that $|z|\ge \ell/2$, we infer
\[
\begin{array}{lcl}
|\tilde{s}^{\eps-}(\tau)-\tilde{s}^{\eps-}(\tau')| & = & \left|\dsp\int_{\Sigma^{\ell}} \frac{\partial (\tilde{u}^{\eps}(\tau)-\tilde{u}^{\eps}(\tau'))}{\partial\nu}\,\overline{w^{-}}-  (\tilde{u}^{\eps}(\tau)-\tilde{u}^{\eps}(\tau'))\frac{\partial\overline{w^{-}}}{\partial\nu}\,d\sigma \right|\\[18pt]
& = & \left|\dsp\int_{\Sigma^{\ell}} \frac{\partial (\tilde{U}^{\eps}(\tau)-\tilde{U}^{\eps}(\tau'))}{\partial\nu}\,\overline{w^{-}}-  (\tilde{U}^{\eps}(\tau)-\tilde{U}^{\eps}(\tau'))\frac{\partial\overline{w^{-}}}{\partial\nu}\,d\sigma \right|\\[13pt]
 & \le & C\,|\tau-\tau'|.
\end{array}
\]
We emphasize that the constant $C>0$ appearing in the last inequality above is independent of $\eps\in(0;\eps_0]$, $\tau$, $\tau'\in\mrm{B}^{\vartheta}_3(O)$. This ends to prove Estimate (\ref{MainEstimate}).

\section*{Acknowledgments}

The research of L. C. was supported by the FMJH through the grant ANR-10-CAMP-0151-02 in the ``Programme des Investissements
d'Avenir''. The research of S.A. N. was supported by the Russian Foundation for Basic Research, grant No. 15-01-02175.

\bibliography{Bibli}
\bibliographystyle{plain}
\end{document}